\pdfoutput=1

\documentclass[12pt,reqno]{amsart}
\usepackage{lmodern} 
\usepackage[scale=0.89]{tgheros} 
\usepackage[osf]{Baskervaldx} 
\usepackage[baskervaldx,cmintegrals,bigdelims,vvarbb]{newtxmath} 
\usepackage[cal=boondoxo]{mathalfa} 
    
  \usepackage{amsmath}          
  \usepackage{amsfonts}           
  \usepackage{amsthm}
  \usepackage{amscd}
  \usepackage{enumerate}
  \usepackage[active]{srcltx}
  \usepackage[percent]{overpic}
  \usepackage[mathscr]{eucal}
  \usepackage{graphicx,epsfig} 
  \usepackage{tikz-cd}        
  \usepackage{color}
  \usepackage{xcolor}
  \usepackage{marginnote}
  \usepackage{hyperref}
  \usepackage{bm,bbm}
  \usepackage{enumitem} 
  \usepackage{tikz-cd}

\tikzset{>=latex,auto}
\usetikzlibrary{positioning,shapes.geometric}

\newtheorem{theorem}{Theorem}[section]
\newtheorem{lemma}[theorem]{Lemma}
\newtheorem{corollary}[theorem]{Corollary}

\theoremstyle{definition}

\theoremstyle{remark}
\newtheorem{remark}[theorem]{Remark}

\theoremstyle{thmx}
\newtheorem{thmx}{Theorem}

\numberwithin{equation}{section}

\newcommand{\thmref}[1]{Theorem~\ref{#1}}

\newcommand{\remref}[1]{Remark~\ref{#1}}

\newcommand{\lemref}[1]{Lemma~\ref{#1}}

\newcommand{\corref}[1]{Corollary~\ref{#1}}
\newcommand{\figref}[1]{Fig.~\ref{#1}}

\newcommand{\vs}{\vspace{6pt}}

\newcommand{\CC}{{\mathbb C}}

\newcommand{\RR}{{\mathbb R}}
\newcommand{\HH}{{\mathbb H}}

\newcommand{\ZZ}{{\mathbb Z}}

\newcommand{\DD}{{\mathbb D}}

\newcommand{\sO}{{\mathscr O}}

\newcommand{\sC}{{\mathscr C}}

\newcommand{\con}{\operatorname{const.}}

\newcommand{\ve}{\epsilon}
\newcommand{\es}{\emptyset}
\newcommand{\sm}{\smallsetminus}
\newcommand{\bd}{\partial}
\newcommand{\ov}{\overline}

\newcommand{\resit}{\operatorname{r\acute{e}sit}}
\newcommand{\res}{\operatorname{res}}

\newcommand{\myre}{\operatorname{Re}}
\newcommand{\myim}{\operatorname{Im}}
\newcommand{\Log}{\operatorname{Log}}

\newcommand{\eps}{\epsilon}
\newcommand{\de}{\delta}
\newcommand{\De}{\Delta}
\newcommand{\om}{\omega}
\newcommand{\Om}{\Omega}
\newcommand{\la}{\lambda}
\newcommand{\La}{\Lambda}
\newcommand{\ga}{\gamma}
\newcommand{\Ga}{\Gamma}

\newcommand{\Mu}{\operatorname{M}}

\newcommand{\e}{\mathrm{e}}
\newcommand{\ii}{\mathrm{i}}

\makeatletter
\newcommand{\oset}[3][0ex]{%
  \mathrel{\mathop{#3}\limits^{
    \vbox to#1{\kern-3\ex@
    \hbox{$\scriptstyle#2$}\vss}}}}
\makeatother

\newcommand{\diam}{\operatorname{diam}}

\newcommand{\Aut}{\operatorname{Aut}}

\DeclareMathAlphabet{\mathlib}{OT1}{LinuxLibertineT-OsF}{m}{it}

\definecolor{myblue}{rgb}{0.2,0.6,0.8}
\definecolor{mygreen}{cmyk}{75.5,0,93.5,0}

\newcommand{\bit}{\it \bfseries}

\begin{document}

\title[Buff forms and invariant curves]{Buff forms and invariant curves of near-parabolic maps}

\author[C. L. Petersen and S. Zakeri]{Carsten Lunde Petersen and Saeed Zakeri}

\address{Department of Mathematical Sciences, University of Copenhagen, Universitetsparken 5, DK-2100 Copenhagen {{\O}}, Denmark}

\email{lunde@math.ku.dk}

\address{Department of Mathematics, Queens College of CUNY, 65-30 Kissena Blvd., Queens, New York 11367, USA} 
\address{The Graduate Center of CUNY, 365 Fifth Ave., New York, NY 10016, USA}

\email{saeed.zakeri@qc.cuny.edu}

\date{\today}

\maketitle

\begin{abstract}
We introduce a general framework to study the local dynamics of near-parabolic maps using the meromorphic $1$-form introduced by X.~Buff. As a sample application of this setup, we prove the following tameness result on invariant curves of near-parabolic maps: Let $g(z)=\la z+O(z^2)$ have a non-degenerate parabolic fixed point at $0$ with multiplier $\la$ a primitive $q$th root of unity, and let $\ga: \, ]-\infty,0] \to \DD(0,r)$ be a $g^{\circ q}$-invariant curve landing at $0$ in the sense that $g^{\circ q}(\ga(t))=\ga(t+1)$ and $\lim_{t \to -\infty} \ga(t)=0$. Take a sequence $g_n(z)=\la_n z+O(z^2)$ with $|\la_n|\neq 1$ such that $g_n \to g$ uniformly on $\DD(0,r)$ and suppose each $g_n$ admits a $g_n^{\circ q}$-invariant curve $\ga_n: \, ]-\infty,0] \to \CC$ such that $\ga_n \to \ga$ uniformly on the fundamental segment $[-1,0]$. If $\la_n^q \to 1$ non-tangentially, then $\ga_n$ lands at a repelling periodic point near $0$, and $\ga_n \to \ga$ uniformly on $]-\infty,0]$. In the special case of polynomial maps, this proves Hausdorff continuity of external rays of a given periodic angle when the associated multipliers approach a root of unity non-tangentially.    
\end{abstract}

\setcounter{tocdepth}{1}

\tableofcontents

\section{Introduction}
When we study parabolic bifurcations we are up against changes on such tiny scales that they disappear before our eyes. Consider for example an analytic map $f_0(z) = z+z^2+O(z^3)$ for which the fixed point equation $f_0(z)=z$ has a double root at $0$. A generic analytic perturbation $f_\eps(z) = f_0(z)+\eps(z)$ has two nearby fixed points $z_1(\eps), z_2(\eps)$ with multipliers close to $1$, i.e., $z_1(\eps), z_2(\eps)\to 0$ and $f'_\eps(z_1(\eps)), f'_\eps(z_2(\eps))\to 1$ as $\ve \to 0$. What does the dynamics of $f_\eps$ look like near $0$? The theory of parabolic implosions, originally developed by P.~Lavaurs in his doctoral thesis, studies precisely this question. A main tool in this theory is an extended notion of Fatou coordinates for near-parabolic maps often known as Douady-Fatou coordinates (see \cite{L} and \cite{D}). \vs

A.~Douady suggested that the dynamics of a holomorphic map $f$ near a parabolic map $f_0$ with a fixed point of multiplier $1$ should look like the time-$1$ flow of the vector field $(f(z)-z) \frac{\bd}{\bd z}$. This idea was developed by M.~Shishikura \cite{Sh} and later in the thesis of R.~Oudkerk \cite{O} written under the supervision of Tan Lei. By its very construction the fixed points of $f$ are also the fixed points of the time-$1$ flow, with matching multiplicities. However, the two maps have different holomorphic indices at each fixed point. X.~Buff came up with the idea of an improved vector field whose time-$1$ flow has a holomorphic index matching that of $f$ at each fixed point, so it mimics the dynamics of $f$ to a higher degree of accuracy. It turns out that primitives (=anti-derivatives) of the dual $1$-forms of Buff vector fields are very close to being Douady-Fatou coordinates in the sense that they conjugate the near-parabolic dynamics to a map which is extremely close to the unit translation. This theme was developed by A.~Ch\'eritat in his habilitation thesis \cite{C} in which he studied parabolic implosions and applications in estimating the conformal radius of quadratic Siegel disks. Our approach has minimal overlap with his work and provides a setup for a range of applications, including the one presented here and those planned for upcoming papers. \vs 

Let us briefly outline this setup; full details are provided in \S \ref{hvff} and \S \ref{bfbp}. Let $U \subset \CC$ be a convex Jordan domain and $f$ be holomorphic in a neighborhood of $\ov{U}$ such that $\myre(f'(z))>0$ for all $z \in U$ and $f(z) \neq z$ for all $z \in \bd U$. Define the {\bit Buff form} of $f$ by
$$
\omega_f = \omega_f(z) \, dz := \frac{f'(z)-1}{(f(z)-z)\Log f'(z)} \, dz,
$$
where $\Log : \CC \sm ]-\infty,0] \to \{ x+iy : |y|<\pi \}$ is the principal branch of the logarithm. The $1$-form $\omega_f$ is non-vanishing and meromorphic in $U$ with finitely many poles at the fixed points of $f$. Moreover, at every fixed point $p=f(p) \in U$, 
$$
\res(\omega_f,p) = \begin{cases} \dfrac{1}{\Log f'(p)} & \quad \text{if} \ \ f'(p) \neq 1 \vs \\ \resit(f,p) & \quad \text{if} \ \ f'(p)=1 \end{cases}
$$
where $\resit(f,p)$ is the {\bit r\'esidu it\'eratif} of $f$ at $p$ (see \S \ref{hiri} and \lemref{res}). The primitive $Z=\phi_f(z) := \int_{z_0}^z \omega_f(s) \, ds$ is a multi-valued function that undergoes the monodromy $Z \mapsto Z+2\pi \ii \res(\om_f,p)$ when $z$ winds once around a fixed point $p$. We call $\phi_f$ the {\bit rectifying coordinate} of $\om_f$ since it transforms $\om_f$ to the $1$-form $dZ$. The collection of all local branches of $\phi_f$ define a Riemann surface with a translation structure. \vs

If $z \in U \cap f^{-1}(U)$ is not a fixed point of $f$, neither is any point of the segment $[z,f(z)] \subset U$ (\lemref{nofp}). Using this fact, the dynamics of $f$ away from the fixed points can be lifted under the rectifying coordinate $Z=\phi_f(z)$ to the new dynamics $F: Z \mapsto Z+1+u_f$, where 
$$
u_f(z):= -1+\int_{[z,f(z)]} \om_f(s) \, ds
$$ 
measures the deviation of $F$ from being the unit translation. The holomorphic function $u_f$, defined on the set of points in $U \cap f^{-1}(U)$ that are not fixed under $f$, extends holomorphically to zero across each fixed point $p=f(p)$. Moreover, if $p$ has multiplier $1$ and fixed point multiplicity $q+1 \geq 2$, then $u_f$ has a zero of order $2q$ at $p$ (\lemref{uw}). Thus, when $f$ is restricted to a small neighborhood of a fixed point, its lifted dynamics $F$ is close to the unit translation. More importantly, this property persists under small perturbations of $f$. \vs 

In \S \ref{bfnpm} we apply this setup to perturbations of a holomorphic map $g(z) = \la z + O(z^2)$ with a non-degenerate parabolic fixed point at $0$ whose multiplier $\la$ is a primitive $q$th root of unity. After an initial analytic change of coordinates we may assume
$$
f:= g^{\circ q} = z + z^{q+1} + b z^{2q+1} + O(z^{2q+2}), 
$$
where $b \in \CC$ is the holomorphic index of $f$ at $0$. Take a sequence $g_n(z)=\la_n z + O(z^2)$ of analytic perturbations of $g$ such that $|\la_n| \neq 1$ for all $n$. Let $f, f_n:=g_n^{\circ q}$ be holomorphic in a neighborhood of a small closed disk $\ov{\DD}(0,r_0)$ and $f_n \to f$ uniformly on $\ov{\DD}(0,r_0)$ as $n \to \infty$. The parabolic fixed point of $g$ at $0$ bifurcates into a simple fixed point of $g_n$ at $0$ of multiplier $\la_n=g'_n(0)$ and a $q$-cycle $\sO_n = \{ p_{n,1}, \ldots, p_{n,q} \}$ of multiplier $\mu_n:=(f_n)'(p_{n,j})$ independent of $j$. The points in $\sO_n$ are asymptotically the vertices of the regular $q$-gon formed by the $q$th roots of $1-\la_n^q$. Set 
$$
\La_n := \frac{1}{\Log \lambda_n^q} \qquad \text{and} \qquad \Mu_n := \frac{1}{\Log \mu_n}.
$$
Then $\La_n, \Mu_n \to \infty$, but $\La_n+q \Mu_n \to \resit(f,0)$ as $n \to \infty$. Moreover, we have the following expressions for the multi-valued rectifying coordinates $\phi=\phi_f$ and $\phi_n=\phi_{f_n}$ of the Buff forms $\om=\om_f$ and $\om_n=\om_{f_n}$, respectively:
\begin{align*}
\phi(z) & = -\frac{1}{qz^q} + \resit(f,0) \, \log z + E(z), \\
\phi_n(z) & = \La_n \log z + \Mu_n \sum_{j=1}^q \log(z-p_{n,j}) + E_n(z). 
\end{align*} 
Here the error terms $E,E_n$ are analytic in a neighborhood of $\ov{\DD}(0,r_0)$ with $E(0)=E_n(0)=0$ and $E_n \to E$ uniformly on $\ov{\DD}(0,r_0)$.

\begin{thmx}[Lifted dynamics is a uniform near-translation] \label{A} 
For every $0<\ve<1$ there is an $0<r<r_0$ such that the following holds for all sufficiently large $n$: \vs

If $z \in \DD(0,r)$ is not fixed under $f_n$, the curve $t \mapsto z_{n,t}=(1-t)z+t f_n(z)$ parametrizing the segment $[z,f_n(z)]$ lifts under $\phi_n$ to a curve $t \mapsto Z_{n,t}$ which satisfies 
\begin{equation}\label{vet}
|Z_{n,t}-(Z_{n,0}+t)|<\ve t \qquad (0 \leq t \leq 1).
\end{equation}
Similarly, the curve $t \mapsto \hat{z}_{n,t}=(1-t)z+t f_n^{-1}(z)$ parametrizing $[z,f_n^{-1}(z)]$ lifts under $\phi_n$ to a curve $t \mapsto \hat{Z}_{n,t}$ which satisfies  
\begin{equation}\label{vett}
|\hat{Z}_{n,t}-(\hat{Z}_{n,0}-t)|<\ve t \qquad (0 \leq t \leq 1).
\end{equation}
\end{thmx}

The special case $t=1$ of the above bounds means that for all sufficiently large $n$ the dynamics of $f_n, f_n^{-1}$ in $\DD(0,r)$ can be lifted under $\phi_n$ to maps $F_n, F_n^{-1}$ that satisfy $|F_n(Z)-(Z+1)|<\ve$ and $|F_n^{-1}(Z)-(Z-1)|<\ve$. If the first $k$ iterates of $z$ under $f_n$ or $f_n^{-1}$ stay in $\DD(0,r) \sm  (\{ 0 \} \cup \sO_n )$, it follows that the first $k$ iterates of $Z$ under $F_n$ or $F_n^{-1}$ stay in the cones  
\begin{align}\label{coneC}
\sC^+_\ve(Z) & :=\{ Z+R\e^{\ii \theta}: R>0, |\theta|<\sin^{-1} \ve \} \notag \vs \\
\sC^-_\ve(Z) & :=\{ Z+R\e^{\ii \theta}: R>0, |\theta-\pi|<\sin^{-1} \ve \}, 
\end{align}
respectively. The bounds \eqref{vet} and \eqref{vett} can be interpreted geometrically as saying that the segment $[z,f_n(z)]$ lifts under $\phi_n$ to a curve joining $Z$ and $\hat{Z}=F_n(Z)$ that lies in $\sC^+_\ve(Z) \cap \sC^-_\ve(\hat{Z})$ (see \figref{twocones}). \vs 

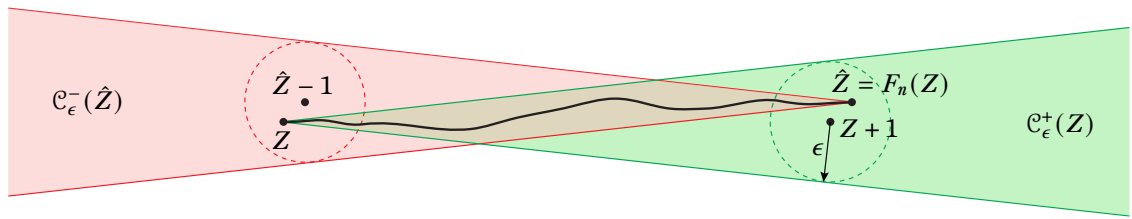
\begin{figure}[]
	\centering
	\begin{overpic}[width=\textwidth]{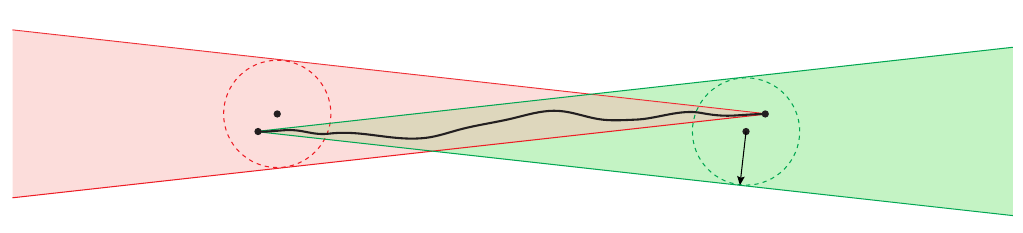}
\put (24.2,7.5) {\footnotesize $Z$}
\put (73.7,8.5) {\footnotesize $Z+1$}
\put (73,12.5) {\footnotesize $\hat{Z}=F_n(Z)$}
\put (24.2,12.3) {\footnotesize $\hat{Z}-1$}
\put (71.3,7) {\footnotesize $\ve$}
\put (90,9) {\footnotesize $\sC^+_\ve(Z)$}
\put (5,11) {\footnotesize $\sC^-_\ve(\hat{Z})$}
	\end{overpic}
\caption{\sl Illustration of \thmref{A}. The segment $[z,f_n(z)]$ lifts under $\phi_n$ to a curve that joins $Z$ and $\hat{Z}=F_n(Z)$ and lies in the intersection of the cones $\sC^+_\ve(Z)$ (in green) and $\sC^-_\ve(\hat{Z})$ (in pink). The $F_n$-orbit of $Z$ and the $F_n^{-1}$-orbit of $\hat{Z}$, as long as they are defined, will stay in the cones $\sC^+_\ve(Z)$ and $\sC^-_\ve(\hat{Z})$, respectively.}
\label{twocones}
\end{figure}

In an ongoing joint project, we set to show that under mild conditions on the multipliers $\la_n$ (roughly, not approaching $\la$ too radially) the orbits of $F_n$ have uniform distance to the orbits of the unit translation, so they stay in horizontal strips rather than cones. This is accomplished by showing that for the normalized Douady-Fatou coordinate $\Phi_n$ which conjugates $F_n$ to the unit translation the distance $\| \Phi_n - \operatorname{id}\|_{\infty}$ is uniformly bounded. Such estimates are of interest in the study of near-parabolic regime and have applications, for example, in understanding the geometry of the limbs of the Mandelbrot set. \vs

In \S \ref{apphl} we use \thmref{A} and the framework provided by Buff forms to study the behavior of invariant curves of near-parabolic maps. To motivate our result, let us first consider the familiar case of polynomial maps and their external rays.  
Let $P_n,P$ be polynomials of degree $d \geq 2$ with connected Julia sets such that $P_n \to P$, and take an angle $\theta \in \RR/\ZZ$ which has period $q$ under multiplication by $d$. Consider the external rays $R_{P_n}(\theta), R_P(\theta)$ landing at $\zeta_n, \zeta$ respectively. These rays come equipped with natural parametrizations by the Green's potential $s \in \, ]0,+\infty[$. Working with the parameter $t=\log s/ \log (d^q)$ instead, we can think of the tails of the external rays $R_{P_n}(\theta), R_P(\theta)$ as arcs $\ga_n, \ga : ]-\infty,0] \to \CC$, with $\lim_{t \to -\infty} \ga_n(t)=\zeta_n$ and $\lim_{t \to -\infty} \ga(t)=\zeta$, which satisfy the invariance relations $P_n^{\circ q}(\ga_n(t))=\ga_n(t+1)$ and $P^{\circ q}(\ga(t))=\ga(t+1)$. Continuity of the B\"{o}ttcher coordinate shows that $\ga_n \to \ga$ uniformly on the fundamental segment $[-1,0]$, hence on every compact subset of $]-\infty, 0]$. It is well known that if $\zeta$ is repelling, then $\ga_n \to \ga$ uniformly on $]-\infty,0]$ and in particular $\zeta_n \to \zeta$. However, when $\zeta$ is parabolic, the behavior of $\ga_n$ can be wild depending on how $P_n$ approaches $P$ in the parameter space of polynomials of degree $d$. In such cases the Hausdorff limit of $\ga_n([-\infty,0])$ can include additional heteroclinic or homoclinic arcs in the Fatou set of $P$. We refer to our recent work \cite{PZ1} for a comprehensive analysis of the structure of possible Hausdorff limits. An emerging theme in that work was that to create wild behavior of $\ga_n$ one often needs the multipliers of the bifurcated cycle of $P_n$ near $\zeta$ to approach a root of unity tangentially (compare \cite[Theorem D]{PZ1}). Here we justify this statement in a general setting beyond  polynomials and external rays. \vs

Once again consider $g_n(z)=\la_n z + O(z^2) \to g(z)=\la z+O(z^2)$ as in the discussion leading to \thmref{A}. Thus, $\la$ is a primitive $q$th root of unity, $|\la_n| \neq 1$, and $f_n=g_n^{\circ q} \to f=g^{\circ q}$ uniformly on $\ov{\DD}(0,r_0)$. Suppose $\ga: \ ]-\infty,0] \to \DD(0,r_0)$ is an $f$-invariant curve that satisfies $f(\ga(t)) = \ga(t+1)$ and $\lim_{t \to -\infty}\ga(t) = 0$. Take a sequence $\ga_n : \ ]-\infty, 0] \to \CC$ of $f_n$-invariant curves satisfying $f_n(\ga_n(t)) = \ga_n(t+1)$ such that $\ga_n \to \ga$ uniformly on compact subsets of $]-\infty,0]$. Recall that $\la_n$ and $\mu_n$ are the multipliers of the fixed point $0$ and the $q$-cycle $\sO_n$ for $g_n$, respectively. It is easy to check that $\la_n^q \to 1$ non-tangentially if and only if $\mu_n \to 1$ non-tangentially as $n \to \infty$. In this case, and after passing to a subsequence, we may assume that either $|\la_n|>1, |\mu_n|<1$ or $|\la_n|<1, |\mu_n|>1$ for all $n$ (see \S \ref{ntma}).    

\begin{thmx}[Tameness]\label{B}
Suppose $\la_n^q \to 1$ (equivalently, $\mu_n \to 1$) non-tangentially as $n \to \infty$. Then \vs
\begin{enumerate}
\item[(i)]
For all sufficiently large $n$ the curve $\ga_n$ lands at a repelling fixed point of $f_n$. More precisely, $\lim_{t \to -\infty}\ga_n(t) = 0$ if $|\la_n|>1$, and $\lim_{t \to -\infty}\ga_n(t) = p_{n,j}$ for some $1 \leq j \leq q$ if $|\la_n|<1$. \vs
\item[(ii)]
As $n \to \infty$, $\ga_n \to \ga$ uniformly on $]-\infty,0]$. In particular, $\ga_n([-\infty,0]) \to \ga([-\infty,0])$ in the Hausdorff metric. 
\end{enumerate}
\end{thmx}

The ``tameness'' in the title refers to the uniform convergence in (ii) which precludes the formation of additional homoclinic arcs in the Hausdorff limit of $\ga_n([-\infty,0])$. Observe that $\ga, g(\ga), \ldots, g^{\circ q-1}(\ga)$ form a $q$-cycle of curves landing at $0$ with a well-defined combinatorial rotation number $0 \leq p/q <1$ with $\gcd(p,q)=1$. They are asymptotic to the $q$ repelling directions of the parabolic point $0$ given by the $q$th roots of unity. When $|\la_n|>1$, the curves $\ga_n, g_n(\ga_n), \ldots, g_n^{\circ q-1}(\ga_n)$ landing at $0$ form a $q$-cycle of the same combinatorial rotation number $p/q$. When $|\la_n|<1$, each point of the $q$-cycle $\sO_n$ remains in a definite sector around one of the repelling directions (see \S \ref{ntma}) and is the landing point of the curve whose unperturbed counterpart lands at $0$ asymptotic to that direction. The proof of \thmref{B}, presented in \S \ref{pfla}, is based on a detailed analysis of the lifted dynamics $F_n$ to construct trapping regions for the curves $\ga_n$ to ensure their landing at the anticipated fixed point and to control their behavior in the Hausdorff metric. \vs 

The motivating example of \thmref{B} is, of course, when $g, g_n$ are restrictions of polynomial maps of a given degree $\geq 2$ and $\ga, \ga_n$ are tails of their external rays of a given periodic angle. Another example is when $g, g_n$ are restrictions of maps in the Eremenko-Lyubich class $\mathcal B$ (transcendental entire functions with bounded singular set) and $\ga, \ga_n$ are tails of their hairs. Notice, however, that we have made no assumption on the smoothness  of either $\ga$ or the $\ga_n$. \vs

\begin{figure}[t]
	\centering
  \includegraphics[width=0.85\linewidth]{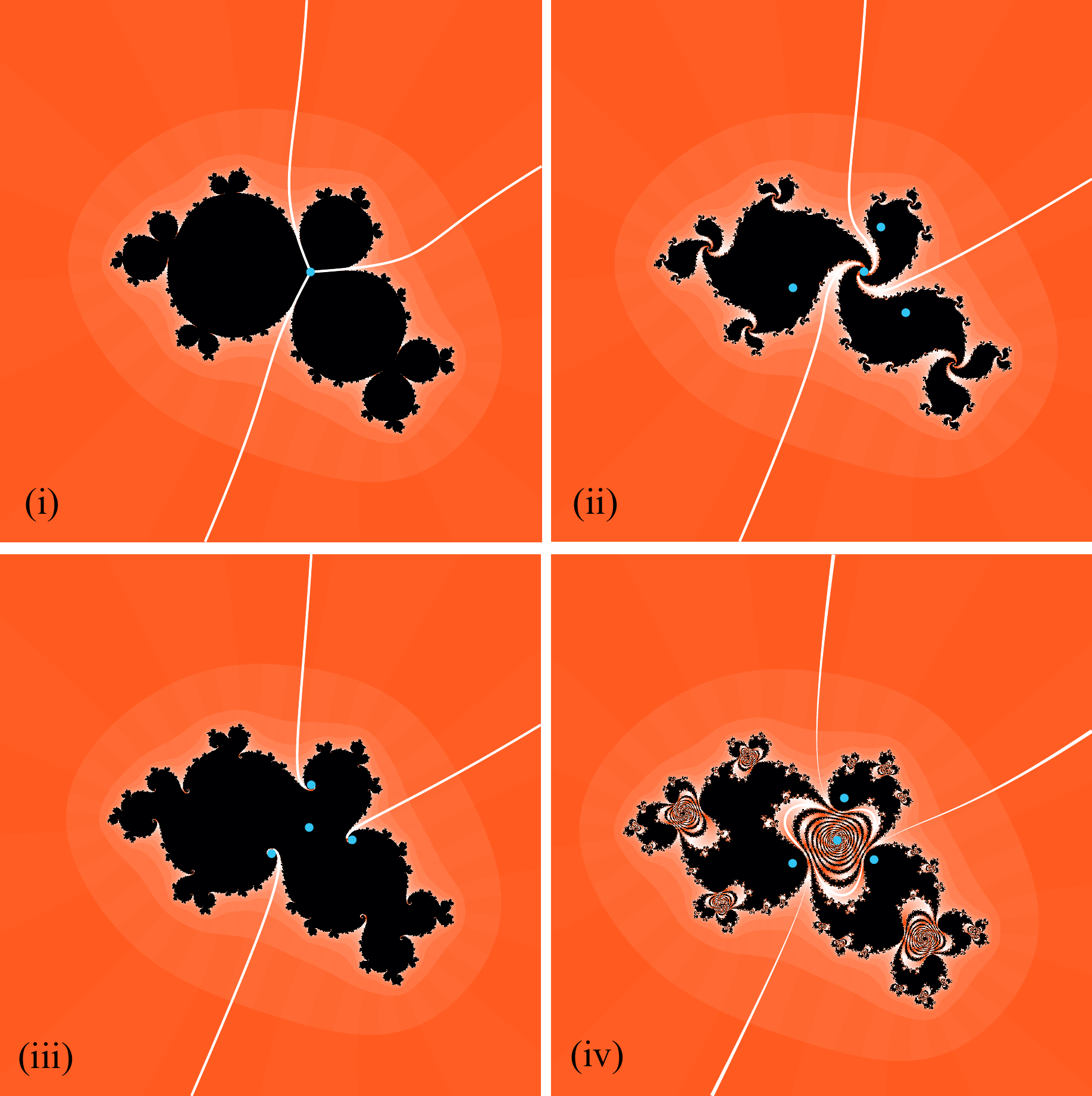}
\caption{\footnotesize Illustration of \corref{nontan}. (i) A cubic polynomial $P_0$ with a parabolic fixed point of multiplier $\exp(2\pi \ii/3)$ at the origin and a cycle of external rays at angles $1/13,3/13,9/13$ landing there. (ii) A perturbation $P_\ve$ with a repelling fixed point at the origin and a nearby attracting $3$-cycle. (iii) A perturbation $P_\ve$ with an attracting fixed point at the origin and a nearby repelling $3$-cycle. In either of cases (ii) or (iii), the repelling orbit of $P_\ve$ captures the cycle of external rays. Assuming the multipliers of the bifurcated cycles of $P_\ve$ tend to $1$ non-tangentially as $\ve \to 0$, these external rays converge in the Hausdorff metric to their counterparts for $P_0$. (iv) Wild behavior of the same external rays for $P_\ve$ when multipliers tend to $1$ tangentially. The Hausdorff limits of these rays will contain a cycle of Hawaiian earrings in the Fatou set of $P_0$.}
\label{poly}
\end{figure}

As a convenient reference, let us restate the polynomial case of \thmref{B} in a more concrete form:
     
\begin{corollary}\label{nontan}
Let $P$ be a polynomial map of degree $d \geq 2$ with connected Julia set. Let $\zeta$ be a non-degenerate parabolic fixed point of $P$, with multiplier a primitive $q$th root of unity, and suppose $\zeta$ is the landing point of a $q$-cycle of external rays $R_P(\theta_1), \ldots, R_P(\theta_q)$. Take a sequence $P_n \to P$ of polynomials of degree $d$ and suppose $P_n$ has a fixed point $\zeta_n$ of multiplier $\lambda_n$ and a $q$-cycle $\sO_n$ of multiplier $\mu_n$ which converge to $\zeta$ as $n \to \infty$. If $\lambda_n^q \to 1$ (equivalently, $\mu_n \to 1$) non-tangentially, then for all sufficiently large $n$ the external rays $R_{P_n}(\theta_1), \ldots, R_{P_n}(\theta_q)$ land at $\zeta_n$ if $|\lambda_n|>1$ and on $\sO_n$ if $|\la_n|<1$. In either case $\ov{R_{P_n}}(\theta_j) \to \ov{R_P}(\theta_j)$ uniformly in the natural Green's potential parametrization, and in particular in the Hausdorff metric.     
\end{corollary}

Compare \figref{poly}. This is the generalization of a result of J. Milnor in \cite{M2} who proved it in the case $d=2$ when $\lambda_n^q,\mu_n \to 1$ radially (i.e., along the real axis). Some topological and measure-theoretic implications of non-tangential multiplier approach for geometrically finite rational maps are studied by C. McMullen in \cite{Mc}. \vs

It is reasonable to speculate that the assumption of non-tangential multiplier approach in \thmref{B} is also necessary for tameness. We will not pursue this issue in the present paper, but in \S \ref{tma} we use the picture of the lifted dynamics to show that if the invariant curves $\ga_n$ eventually leave the neighborhood of $0$ as $t \to -\infty$ (thus creating a heteroclinic arc in the Hausdorff limit), they must pass through one of the $q$ ``gates'' between $0$ and the $q$-cycle $\sO_n$. This statement, along with special cases of the general theory developed in \cite{PZ1} play a role in our recent work on the cubic connectedness locus \cite{PZ2}. \vs   
 
\noindent
{\bf Acknowledgments.} We thank X. Buff for many conversations on this topic over the years. We acknowledge the influence of A.~Ch\'eritat's presentation in his habilitation thesis \cite{C}, where Buff forms make their first appearance in print. The main application presented in this paper is partially inspired by J. Milnor's study of the external rays of near parabolic maps in the quadratic family in \cite{M2}. The first author would like to thank the Danish Council for Independent Research -- Natural Sciences for support via the grant DFF-1026-00267B. The second author acknowledges the partial support of the Research Foundation of The City University of New York via grant TRADB-54-375.

\section{Holomorphic $1$-forms and vector fields}\label{hvff}

\noindent
Throughout the paper the following notation will be adopted: \vs 

\begin{enumerate}
\item[$\bullet$]
$\DD(p,r):=\{ z \in \CC: |z-p|<r \},\ \DD:=\DD(0,1)$
\item[$\bullet$]
$\HH := \{ z \in \CC: \myim(z)>0 \}$
\item[$\bullet$]
$[z,w]$ is the Euclidean segment between $z,w \in \CC$
\item[$\bullet$]
$n \gg 1$ stands for ``all sufficiently large $n$'' \vs
\end{enumerate}

Consider a bounded domain $U \subset \CC$ and assume for simplicity that the boundary $\bd U$ is a smooth Jordan curve. Let $\chi=\chi(z)\, \frac{\bd}{\bd z}$ be a holomorphic vector field defined in a neighborhood of the closure $\ov{U}$, with $\chi \neq 0$ on $\bd U$. We denote by $p_1, \ldots, p_k$ the zeros of $\chi$ in $U$. The dual $1$-form $\om=\om(z) \, dz$ with $\om(z)=1/\chi(z)$ is meromorphic in $U$ with poles at $p_1, \ldots, p_k$. Both $\chi, \om$ are holomorphic and non-vanishing in the punctured domain $U^*:= U \sm \{ p_1, \ldots, p_k \}$.

\subsection{Rectifying coordinates}\label{rc}

In general $\om$ is not the differential of a single-valued
holomorphic function in $U^*$. But one can easily construct primitives of $\om$ by passing to the universal covering as follows. Take a holomorphic universal covering map $h: (\HH,\zeta_0) \to (U^\ast,z_0)$ and let $\tilde{\om}:=h^\ast \om$. Define $\tilde{\phi}: \HH \to \CC$ by 
$$
Z = \tilde{\phi}(\zeta) := \int_{[\zeta_0,\zeta]} \tilde{\om}(s) \, ds.
$$
Evidently $\tilde{\phi}$ is well-defined and holomorphic with $\tilde{\phi}(\zeta_0)=0$, and  
\begin{equation}\label{Zder}
d\tilde{\phi}= \tilde{\om} \, d\zeta = \om \, dz. 
\end{equation}
As $d\tilde{\phi}/d\zeta \neq 0$ everywhere in $\HH$, $\tilde{\phi}$ is locally biholomorphic. If $\sigma_j \in \Aut(\HH)$ is the deck transformation of $h$ represented by the homotopy class in $\pi_1(U^*,z_0)$ with winding number $1$ around $p_j$ and $0$ around the remaining fixed points, then $\tilde{\phi} \circ \sigma_j = T_j \circ \tilde{\phi}$, where $T_j$ is the translation $Z \mapsto Z + 2\pi \ii \, \res(\om,p_j)$. \vs

In practice, it is convenient to forget about the universal covering $h$ altogether and think of the primitive as a multi-valued holomorphic function in $U^\ast$ that undergoes the monodromy $T_j$ whenever $z$ winds once around $p_j$ counterclockwise. In other words, we have the multi-valued primitive 
$$
Z = \phi(z) := \int_{z_0}^z \om(s) \, ds
$$
whose local branches near $z$ correspond to the choice of homotopy classes of paths in $U^*$ from $z_0$ to $z$. Following the terminology of classical complex analysis, we call each value of $Z$ a {\bit lift} of $z$ under $\phi$. The collection of all local branches of $\phi$ define coordinates for a {\it translation structure} on $U^*$ since any two branches differ by an element of the group generated by the translations $T_1, \ldots, T_k$. \vs

By \eqref{Zder} each branch of $Z=\phi(z)$ transforms the $1$-form $\om$ to $dZ$. Equivalently, it transforms the vector field $\chi$ to the unit vector field $\frac{\bd}{\bd Z}$. In fact, if $t \mapsto z(t)$ is the complex-time parametrization of a solution of $dz/dt = \chi(z(t))$, then  
$$
\frac{d}{dt}(\phi(z(t))) = \frac{d\phi}{dz}(z(t))\, \frac{dz}{dt}(t) = \om(z(t))\, \chi(z(t)) = 1. 
$$
For this reason, we often refer to the branches of $\phi$ as the {\bit rectifying coordinates} for $\om$ or $\chi$.\footnote{Another interpretation for $Z=\phi(z)$ is the ``complex time'' it takes for the flow of $\chi$ to map the initial condition $z_0$ to the point $z$.} 

\subsection{Real-time trajectories of $\chi$}\label{rtt}
 
Under a rectifying coordinate $Z=\phi(z)$ the real-time trajectories of $\chi$ map to horizontal lines in $\CC$. More generally, for any $\alpha \in \CC$ with $|\alpha|=1$ the coordinate $Z=\phi(z)$ transforms the rotated vector field $\alpha \, \chi$ to $\alpha \frac{\bd}{\bd Z}$, so the real-time trajectories of $\alpha \, \chi$ map to straight lines in $\CC$ parallel to $\alpha$. \vs

Here is a basic example: When $\chi=z \, \frac{\bd}{\bd z}$ we have $\phi(z)=\log z$ under which the real-time trajectories of $\chi$ (the radial lines emanating from $0$) map to horizontal lines in $\CC$. For the rotated vector field $\ii \, \chi$, the trajectories (concentric circles centered at $0$) map under $\phi$ to vertical lines in $\CC$. \vs    

Suppose the zero $p=p_j$ has multiplicity $m=m_j \geq 1$ and $c:=\res(\om,p)$. If $m=1$, then $c \neq 0$ and there is a local analytic change of coordinate $w=w(z)$ near $p$, with $w(p)=0$, which transforms $\chi$ to the linear normal form
$$
\chi_0 = \frac{1}{c} w \ \frac{\bd}{\bd w}.
$$  
If $m \geq 2$, then there is a local analytic change of coordinate $w$ which transforms $\chi$ to the normal form  
\begin{equation}\label{m>1}
\chi_0 = \frac{w^m}{1+cw^{m-1}} \ \frac{\bd}{\bd w} 
\end{equation}
(see for example \cite[Theorem 5.25]{IY}). It follows in particular that in either case the Poincar\'e-Hopf index of $\chi$ at $p$ is $m$ and therefore it is always a positive integer. \vs

The vector field $\chi_0$ in \eqref{m>1} has the rectifying coordinate
\begin{equation}\label{phiformula}
\phi_0(w) = -\frac{1}{m-1} \, w^{-(m-1)} + c \log w. 
\end{equation}
It is not hard to see from this explicit formula that there is a neighborhood of $0$ in which every real-time trajectory of $\chi_0$ converges to $0$ in either positive or negative time (or both). More precisely, they converge to $0$ asymptotically in the direction of an $(m-1)$st root of $1$ as $t \to -\infty$, and in the direction of an $(m-1)$st root of $-1$ as $t \to +\infty$ (see \figref{vfp}). By \eqref{m>1} the vector field $\chi_0$ and therefore its phase portrait is invariant under the rotation $w \mapsto \e^{2\pi \ii/(m-1)} w$. It follows from this local picture that any trajectory of $\chi_0$ that accumulates on $0$ in positive or negative time has to converge to $0$. Going back to the $z$-plane, we conclude that if $\ga: \, [0,+\infty[ \to U^*$ is a trajectory of $\chi$ accumulating at $p_j$ in positive time, then $\lim_{t \to +\infty} \ga(t)=p_j$. Similarly, if $\ga: \, ]-\infty,0] \to U^*$ is a trajectory accumulating at $p_j$ in negative time, then $\lim_{t \to -\infty} \ga(t)=p_j$. \vs

\begin{figure}[t]
	\centering
\begin{minipage}{0.32\textwidth}
  \centering
  \includegraphics[width=1.1\linewidth]{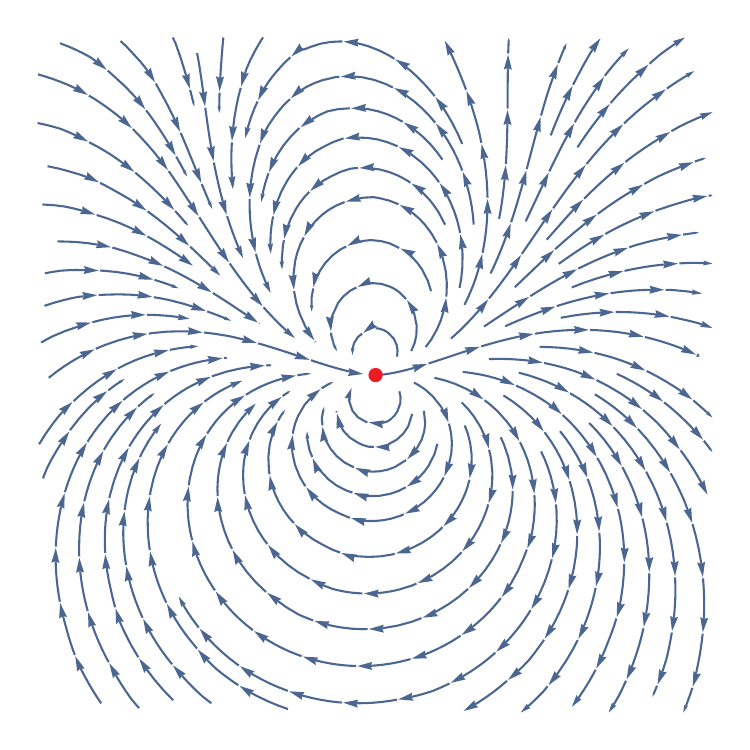}
\end{minipage}
\begin{minipage}{0.32\textwidth}
  \centering
  \includegraphics[width=1.1\linewidth]{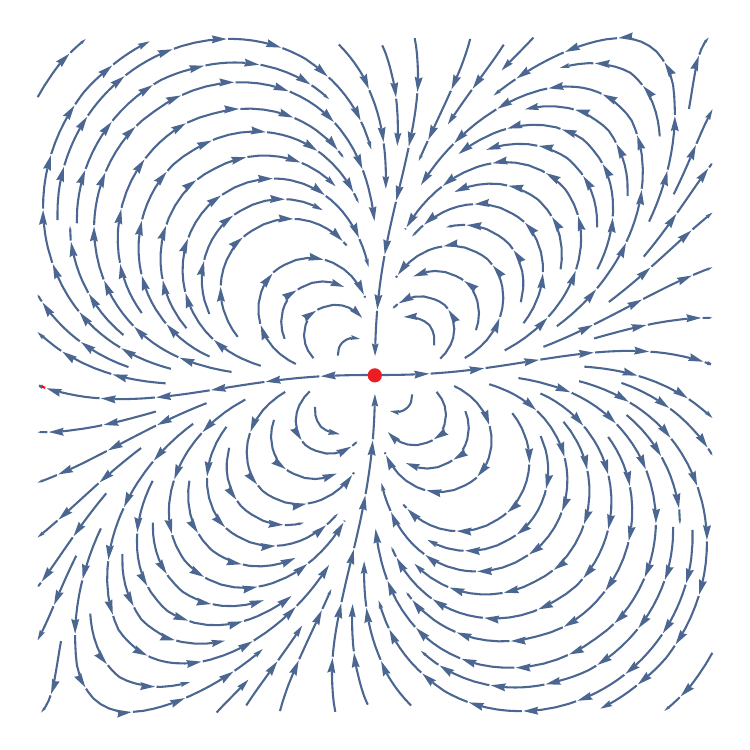}
\end{minipage}	
\begin{minipage}{0.32\textwidth}
  \centering
  \includegraphics[width=1.1\linewidth]{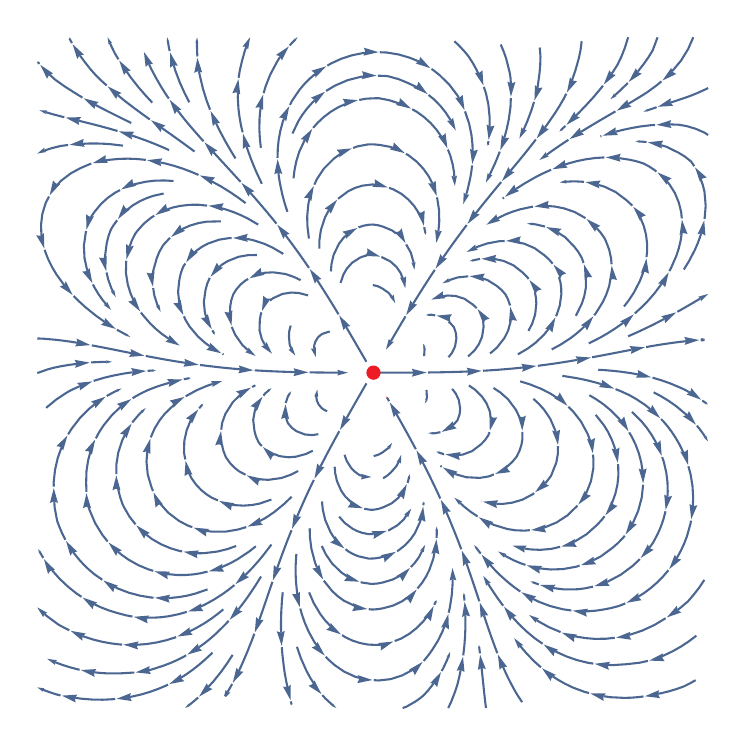}
\end{minipage}	
\caption{\footnotesize Real-time trajectories of the vector field $w^m/(1+cw^{m-1}) \, \frac{\bd}{\bd w}$ near the singular point at $w=0$, where $c=0.2+i$. From left to right we have shown the cases $m=2,3,4$. Each phase portrait is invariant under rotation by $1/(m-1)$ of a turn around $0$.}
\label{vfp}
\end{figure}

One final remark: Suppose $\ga : \, ]a,b[ \to U^*$ is a maximally extended trajectory of $\chi$. If $a \neq -\infty$, then $\lim_{t \to a^+} \ga(t)$ exists and belongs to $\bd U$ and $\ga$ extends to a trajectory of $\chi$ in some neighborhood of $a$. Similarly, if $b \neq +\infty$, then $\lim_{t \to b^-} \ga(t)$ exists and belongs to $\bd U$ and $\ga$ extends to a trajectory of $\chi$ in some neighborhood of $b$. These statements follow easily from our assumption that $\chi \neq 0$ on $\bd U$.

\subsection{Metric and curvature} \label{ometric}

The line element $|\om|=|\omega(z)| \, |dz|$ defines a conformal metric in $U^*$ which is Euclidean in the sense that its Gaussian curvature is identically $0$. Each rectifying coordinate $\phi$ is a local isometry between $(U^*, |\om|)$ and $(\CC, |dZ|)$. In particular, the real-time trajectories of $\chi$ are $|\om|$-geodesics in $U^*$. More generally, for any $\alpha \in \CC$ with $|\alpha|=1$ the real-time trajectories of $\alpha \chi$ are $|\om|$-geodesics in $U^*$. \vs

The metric $|\om(z)| \, |dz|$ is not complete since $\bd U$ is at finite distance from every point in $U^*$. However, the singularities $p_1, \ldots, p_k$ are infinitely far from any point of $U^*$. This follows from the fact that $|\om(z)| \asymp |z-p_j|^{-m_j}$ for $z$ near $p_j$, where $m_j \geq 1$ is the multiplicity of $p_j$ as a zero of $\chi$. \vs

For future reference in \S \ref{ldfn}, let us calculate the $|\om|$-curvature of a positively oriented circle $z(t)=r\e^{\ii t}$ assuming this circle is contained in $U^*$. In a rectifying coordinate $\phi$, this is the same as the Euclidean curvature $\kappa=\kappa(t)$ of the image curve $Z(t):=\phi(z(t))$. The tangent vector $v=v(t)$ to $Z(t)$ can be expressed as  
$$
v = \frac{dZ}{dt} = \frac{d\phi}{dz} \ \frac{dz}{dt} = \ii \, z \om(z).
$$
Hence,
\begin{align} \label{kappa} 
\kappa & = \frac{d \arg(v)}{d(\text{arc-length})} = \frac{1}{|v|} \, \frac{d \arg(v)}{dt} = \frac{1}{|z \om(z)|} \, \frac{d}{dt}(\myim(\log (\ii \, z \om(z))))  \notag \\
& = \frac{1}{|z \om(z)|} \, \myim \Big[ \frac{d}{dt}(\log (\ii \, z \om(z))) \Big] = \frac{1}{|z \om(z)|} \, \myim \Big[ \frac{\bd}{\bd z}(\log (\ii \, z \om(z))) \ \frac{dz}{dt} \Big] \notag \\
& =  \frac{1}{|z \om(z)|} \, \myim \Big[ \frac{\om(z)+z\om'(z)}{z\om(z)} \cdot \ii \, z \Big] = \frac{1}{|z \om(z)|} \, \myre \Big[ 1 + \frac{z\om'(z)}{\om(z)} \Big]. \vs 
\end{align}

\subsection{Periodic orbits of $\chi$}

Let $\ga$ be a closed trajectory of $\chi$ in $U$ of period $\tau>0$. By the Poincar\'e-Hopf theorem the Jordan domain bounded by $\ga$ contains a unique singularity $p$ of index (=multiplicity) $1$, where $\res(\om,p)=\tau/(2\pi \ii)$ is purely imaginary. The time-$\tau$ flow of $\chi$ is a holomorphic function which restricts to the identity map on $\ga$, hence in the maximal neighborhood $W \subset U$ of $\ga$ in which it is defined. It is not hard to check that $W$ is a topological disk and $W \sm  \{ p \}$ is foliated by closed trajectories of period $\tau$ (compare \cite[\S 9]{St} for a similar situation in the context of quadratic differentials). Any local branch of the rectifying coordinate $Z=\phi(z)$ can be continued analytically in $W \sm  \{ p \}$ which yields a multi-valued function onto a half-plane of form $\{ Z : \myim(Z)>Y_0 \}$. The inverse  
$$
\psi:=\phi^{-1}: \{ Z : \myim(Z)>Y_0 \} \to W \sm  \{ p \} 
$$
defines an infinite cyclic holomorphic covering map with the deck group generated by the translation $Z \mapsto Z+\tau$. The obstruction to enlarging $W$ further can be only $\bd W$ meeting the boundary $\bd U$ or a singularity in $U$. The latter cannot happen since all points on the closed trajectory $\psi ( \{ Z: \myim(Z)=Y_0+1 \})$ are at $|\om|$-distance $1$ to $\bd W$ and $+\infty$ to the singular points in $U$ (see \S \ref{ometric}). Our standing assumption that $\chi$ is defined in a neighborhood of $\ov{U}$ and is non-zero on $\bd U$ implies that $\bd W$ itself is a closed trajectory of $\chi$ of the same period $\tau$. It follows that the smooth Jordan curves $\bd W$ and $\bd U$ intersect and they are tangent to each other wherever they meet.  \vs

Conversely, suppose $p$ is a singular point of $\chi$ in $U$ with multiplicity $m=1$ such that $c=\res(\om,p) \neq 0$ is purely imaginary. Then by the discussion in \S \ref{rtt} there is a neighborhood of $p$ in $U$ in which $\chi$ is analytically transformed to $(w/c) \frac{\bd}{\bd w}$, so the real-time trajectories of $\chi$ are closed with period $\tau=2\pi \ii \, c$. By the preceding paragraph there is a largest flow-invariant Jordan domain neighborhood $W \subset U$ of $p$ such that $W \sm \{ p \}$ is foliated by closed trajectories of the same period $\tau$. Moreover, $\bd W \cap \bd U \neq \es$ and the two Jordan curves are tangent wherever they meet. We call $W$ the {\bit canonical neighborhood} of $p$ for the vector field $\chi$. Notice that this neighborhood depends on the choice of $U$.

\begin{remark}
As a consequence of the preceding remarks, we see that {\it a holomorphic vector field in $U$ can never have a limit cycle}. It follows from the Poincar\'e-Bendixson theorem \cite{HSD} that every trajectory $\ga : \, ]a,+\infty[ \to U^*$ is either periodic, or $\lim_{t \to +\infty} \ga(t)$ is a singular point. A similar statement holds for every trajectory $\ga : \, ]-\infty,b[ \to U^*$. 
\end{remark}

\subsection{Rotations of vector fields}

Suppose all singularities $p_1, \ldots, p_k$ of $\chi$ in $U$ have multiplicity $1$, so $c_j:=\res(\om,p_j)$ is non-zero for all $1 \leq j \leq k$. Consider the unimodular complex numbers (or unit vectors) 
$$
\alpha_j:= \ii \, \frac{c_j}{|c_j|} \qquad (1 \leq j \leq k).
$$ 
The rotated vector field $\alpha_j \chi$ has a purely imaginary residue at $p_j$, so there is a canonical neighborhood $W_j \subset U$ of $p_j$ for the vector field $\alpha_j \, \chi$ in the sense described above. Each trajectory of $\alpha_j \, \chi$ in $\ov{W_j} \sm \{ p_j \}$ is closed of period $2\pi|c_j|$ and under a rectifying coordinate $Z=\phi(z)$ of $\chi$ they all map to straight lines parallel to $\alpha_j$. There is a $Y_j \in \RR$ for which the inverse $\psi=\phi^{-1} : \{ Z: \myim(Z/\alpha_j) > Y_j \} \to W_j \sm \{ p_j \}$ is an infinite cyclic covering map whose deck group is generated by the translation $Z \mapsto Z+2\pi \ii \, c_j$.             

\begin{lemma}\label{disjoint}
Let $p_i, p_j$ be distinct singularities of $\chi$ in $U$ and $\gamma_i \subset \ov{W_i}$ and $\gamma_j \subset \ov{W_j}$ be closed trajectories of the corresponding rotated fields $\alpha_i \chi$ and $\alpha_j \chi$. Then $\gamma_i \cap \gamma_j = \es$.  
\end{lemma}

\begin{proof}
First we note that $\gamma_i, \gamma_j$ cannot be tangent to each other at any point. If this were to happen at some point $z$, we would have $\alpha_i=\alpha_j$ or $\alpha_i=-\alpha_j$. In the first case $\gamma_i, \gamma_j$ would be two different trajectories of the same vector field $\alpha_i \, \chi$ passing through $z$, which is impossible by the uniqueness of solutions of differential equations. In the second case, $\gamma_i$ and the reverse curve $t \mapsto \gamma_j(-t)$ would be two different trajectories of $\alpha_i \, \chi$ passing through $z$, which is again impossible. Thus, $\gamma_i$ and $\gamma_j$ can only intersect transversally. Suppose they meet transversally at some $z$ where $\gamma_i$ crosses $\gamma_j$ from, say, left to right in the sense that the pair $(\alpha_i \chi(z), \alpha_j \chi(z))$ forms a positive basis for $\RR^2$. Since $\gamma_i, \gamma_j$ are Jordan curves, they must meet transversally at another point $z'$ where $\gamma_i$ crosses $\gamma_j$ from right to left in the sense that $(\alpha_i  \chi(z), \alpha_j \chi(z))$ forms a negative basis for $\RR^2$. Looking at the pair of vectors at $z$ gives $\myim(\alpha_j/\alpha_i)>0$ while at $z'$ we obtain $\myim(\alpha_j/\alpha_i)<0$, a contradiction.
\end{proof}

\begin{corollary}\label{jcs}
Suppose all singularities $p_1, \ldots, p_k$ of $\chi$ in $U$ have multiplicity $1$. Consider the unimodular complex numbers $\alpha_1, \ldots, \alpha_k$ as above and let $W_j \subset U$ be the canonical neighborhood of $p_j$ for the rotated vector field $\alpha_j \, \chi$. Then the closures $\ov{W_1}, \ldots, \ov{W_k}$ are pairwise disjoint, while $\bd W_j$ meets $\bd U$ tangentially for all $1 \leq j \leq k$.
\end{corollary}

\section{Buff forms and their basic properties}\label{bfbp}

\subsection{Holomorphic index and r\'esidu iteratif}\label{hiri} 

We begin by briefly recalling some basic definitions; for details we refer to the paper of X. Buff and A. Epstein \cite{BE} (see also \cite{M1} and \cite{C}). Let $p$ be an isolated fixed point of a holomorphic map $f$, so $f(z)=p + \lambda (z-p) + O((z-p)^2)$ in some neighborhood of $p$. Here $\lambda:=f'(p)$ is the {\bit multiplier} of the fixed point $p$. The {\bit holomorphic index} of $f$ at $p$ is defined by   
$$
\iota(f,p) := \res \Big(\frac{dz}{z-f(z)}, p \Big) = \int_{|z-p|=\ve} \frac{dz}{z-f(z)},   
$$ 
which is invariant under analytic change of coordinates. It is easy to check that   
\begin{equation}\label{eq:index}
\iota(f,p) = \frac{1}{1-\lambda} \qquad \text{if} \ \lambda \neq 1.
\end{equation}
A closely related quantity is the {\bit r\'esidu it\'eratif} of $f$ at $p$ defined by
$$
\resit(f,p) := \frac{m}{2} - \iota(f,p),  
$$ 
where $m \geq 1$ is the fixed point multiplicity of $p$, i.e., the multiplicity of $p$ as a root of $z-f(z)=0$. In particular,
\begin{equation}\label{eq:resit}
\resit(f,p)=\frac{1}{2} - \frac{1}{1-\lambda} \qquad \text{if} \ \lambda \neq 1.  
\end{equation}
Thus, a simple fixed point $p$ is repelling or attracting according as $\resit(f,p)$ has positive or negative real part. \vs

When $\lambda=1$ the formulas \eqref{eq:index} and \eqref{eq:resit} for the holomorphic index and r\'esidu it\'eratif are no longer valid. To calculate them in this case, suppose $p$ has multiplier $1$ and fixed point multiplicity $m \geq 2$. Then there is an analytic change of coordinates in which $f$ has the local normal form
$$
f(z)=z+a(z-p)^m+b(z-p)^{2m-1}+O((z-p)^{2m})
$$ 
with $a,b \in \CC$ and $a \neq 0$. A straightforward computation then shows that 
\begin{equation}\label{resitcomp}
\iota(f,p)=\frac{b}{a^2}, \qquad \text{so} \qquad \resit(f,p)=\frac{m}{2}-\frac{b}{a^2}. 
\end{equation}
For a generic perturbation $f_\ve$ of $f$ the parabolic fixed point $p$ splits into $m$ simple fixed points $p_1(\ve), \ldots,  p_m(\ve)$ of multipliers $\lambda_1(\ve), \ldots, \lambda_m(\ve)$. Continuity of the holomorphic index then shows that 
$\lim_{\ve \to 0} \sum_{j=1}^m 1/(1-\lambda_j(\ve)) = \iota(f,p)$, or
\begin{equation}\label{resitcont}
\lim_{\ve \to 0} \sum_{j=1}^m \resit(f_\ve,p_j(\ve)) = \resit(f,p).
\end{equation}

When $p$ is parabolic with multiplier $\lambda$ a primitive $q$th root of unity, we can apply the preceding remarks to the iterate $f^{\circ q}$. In this case the multiplicity of $p$ as a fixed point of $f^{\circ q}$ is necessarily of the form $m=\nu q+1$ for some integer $\nu \geq 1$ called the {\bit degeneracy order} of $p$, the case $\nu=1$ being considered non-degenerate. Geometrically, there are $m-1=\nu q$ immediate parabolic basins of $f$ attached to $p$, and these fall into $\nu$ disjoint cycles of length $q$. 

\subsection{Buff forms and their residues}
 
For simplicity, we will make the following standing assumptions for the remainder of this section: 
\begin{enumerate}
\item[$\bullet$]
$U \subset \CC$ is a convex domain bounded by a smooth Jordan curve (in our applications $U$ can be taken to be a Euclidean disk). 
\item[$\bullet$] 
$f$ is holomorphic in a neighborhood of $\ov{U}$ and $\myre(f')>0$ in $U$. 
\item[$\bullet$]
$f(z) \neq z$ for all $z \in \bd U$. 
\end{enumerate} 
Define the {\bit Buff form} of $f$ by
$$
\omega_f = \omega_f(z) \, dz := \frac{f'(z)-1}{(f(z)-z)\Log f'(z)} \, dz.
$$
Here $\Log : \CC \sm ]-\infty,0] \to \{ x+iy : |y|<\pi \}$ is the principal branch of the logarithm. The $1$-form $\omega_f$ is non-vanishing and meromorphic in $U$ with finitely many poles at the fixed points of $f$. The dual point of view is provided by the {\bit Buff vector field} 
$$
\chi_f = \chi_f(z) \, \frac{\bd}{\bd z} := \frac{(f(z)-z)\Log f'(z)}{f'(z)-1} \, \frac{\bd}{\bd z}
$$
which is holomorphic in $U$ with finitely many zeros at the fixed points of $f$. 

\begin{lemma}\label{res}
For every fixed point $p=f(p) \in U$, 
$$
\res(\omega_f,p) = \begin{cases} \dfrac{1}{\Log f'(p)} & \quad \text{if} \ \ f'(p) \neq 1 \vs \\ \resit(f,p) & \quad \text{if} \ \ f'(p)=1. \end{cases}
$$
\end{lemma}

\begin{proof}
This can be found in \cite[\S 1.3]{C}. For convenience we give a short proof here; this will also provide local expansions for later use. Let $p=0$ and $\lambda:=f'(0)$. If $\lambda \neq 1$, then 
$$
\omega_f(z)=\frac{(\lambda-1)+O(z)}{((\lambda-1)z+O(z^2)) \Log(\lambda+O(z))} = \frac{1}{z \Log \lambda}+O(1),  
$$ 
which shows $\res(\omega_f,0)=1/\Log \lambda$. Suppose now that $0$ has multiplier $\lambda=1$ and fixed point multiplicity $m=q+1 \geq 2$, so 
\begin{equation}\label{EP}
f(z)=z+\De(z) \quad \text{where} \quad \De(z)=az^{q+1}+O(z^{q+2}) \quad \text{with} \quad a \neq 0. 
\end{equation}
Using the Taylor expansion 
$$
\frac{z}{\Log(1+z)}=1+\frac{1}{2} z - \frac{1}{12} z^2 + O(z^3) \qquad \text{as} \ z \to 0,
$$
we obtain the following expression for the Buff form of $f$:
\begin{equation}\label{g}
\omega_f (z) = \frac{\De'(z)}{\De(z) \log(1+\De'(z))} = \frac{1+\frac{1}{2}\De'(z)-\frac{1}{12}\De'(z)^2}{\De(z)} + O(z^{2q-1}). 
\end{equation} 
Since the term $\De'(z)^2/\De(z)$ is holomorphic near $0$, we conclude that 
\begin{align*}
\res(\omega_f,0) & = \res \Big( \frac{1}{\De},0 \Big) + \frac{1}{2} \res \Big( \frac{\De'}{\De},0 \Big) \\
& = -\iota(f,0) + \frac{q+1}{2} = \resit(f,0). \qedhere 
\end{align*}
\end{proof}

\begin{remark}
The Buff form is not invariant under local analytic conjugation, but it is asymptotically invariant. More precisely, if $w=h(z)$ is a local analytic change of coordinates near $0$ which conjugates the parabolic germs  
$$
f(z)=z+az^{q+1}+O(z^{q+2}) \quad \text{and} \quad g(w)=w+bw^{q+1}+O(w^{q+2}) \quad (a,b \neq 0)
$$
in the sense that $g \circ h = h \circ f$, then $
h^*(\omega_g)(z) - \omega_f(z) = O(z^q)$. 
\end{remark}

\subsection{The lifted dynamics}
When suitably lifted under the rectifying coordinate $\phi_f$ of the Buff form $\om_f$, the dynamics of $f$ becomes a new dynamics $F$ with a rather simple description. To make this idea precise, we start with the following lemma which asserts that under mild conditions the Euclidean segment $[z,f(z)]$ does not contain any fixed points of $f$. Set $V:=U \cap f^{-1}(U)$, $U^\ast := \{ z \in U: f(z) \neq z \}$, and $V^\ast := \{ z \in V: f(z) \neq z \}$.   
 
\begin{lemma}\label{nofp}
If $z \in V^\ast$ then $[z,f(z)] \subset U^\ast$. 
\end{lemma}

\begin{proof}
We have $[z,f(z)] \subset U$ by convexity. Take the parametrization of $[z,f(z)]$ given by $z_t:=(1-t)z+tf(z)$ for $0 \leq t \leq 1$. Let $m :=\min_{[z,f(z)]} \myre(f') >0$. Then, for any $0 \leq a \leq 1$,  
\begin{align*}
|f(z_a)-z_a| & = \left| (f(z)-z)+ \int_{[z,z_a]} (f'(\zeta)-1) \, d\zeta \right| =  |f(z)-z| \ \left| 1+ \int_0^a (f'(z_t)-1) \, dt \right| \\
& \geq |f(z)-z| \left( 1+ \int_0^{a} \myre(f'(z_t)-1) \, dt \right) \\
& \geq |f(z)-z| \ ( 1+ a (m-1)) = |f(z)-z| \ (a m + 1-a) >0. \qedhere 
\end{align*}
\end{proof}

To construct the lift of $f$ carefully, we once again introduce a universal covering map $h: (\HH,\zeta_0) \to (U^\ast,z_0)$, the pull-back $\tilde{\om}_f:=h^\ast \om_f$ and the primitive $\tilde{\phi}_f : \HH \to \CC$ defined by $Z=\tilde{\phi}_f(\zeta)= \int_{[\zeta_0,\zeta]} \tilde{\om}_f(s) \, ds$, as in \S \ref{rc}. Define a map $\tilde{f}: h^{-1}(V^\ast) \to \HH$ as follows. Let $\zeta \in h^{-1}(V^\ast)$ and $z=h(\zeta)$. By \lemref{nofp} the segment $[z,f(z)]$ is contained in $U^\ast$, so it lifts uniquely under $h$ to a path which starts at $\zeta$ and ends at some $\zeta'$. Define $\tilde{f}(\zeta):=\zeta'$. Evidently $\tilde{f}$ satisfies $h \circ \tilde{f} = f \circ h$ and in particular it is holomorphic. Moreover, 
$$
\tilde{\phi}_f(\tilde{f}(\zeta))-\tilde{\phi}_f(\zeta)= \int_{[\zeta,\tilde{f}(\zeta)]} \tilde{\om}_f(s) \, ds = \int_{[z,f(z)]} \om_f(s) \, ds 
$$ 
since $[z,f(z)]$ and $h([\zeta,\tilde{f}(\zeta)])$ are homotopic in $U^\ast$. In other words, the lifted dynamics $\tilde{f}$ in the rectifying coordinate $Z=\tilde{\phi}_f(\zeta)$, or equivalently in the multi-valued rectifying coordinate $Z=\phi_f(z)=\int_{z_0}^z \om_f(s) \, ds$, is seen as the map 
$$
F: Z \mapsto Z+1+u_f,
$$ 
where 
$$
u_f(z):= -1+\int_{[z,f(z)]} \om_f(s) \, ds \qquad (z \in V^\ast).
$$ 
Notice how $u_f$, the difference between $F(Z)$ and $Z+1$, does not depend on the branch of $Z=\phi(z)$ and only depends on $z$. We will see that in sufficiently small neighborhoods of a fixed point of $f$ the function $u_f$ can be made arbitrarily small and therefore the lifted dynamics under $\phi_f$ is seen as a near unit translation. This picture persists under perturbations $f_\ve$ of $f$, so $\phi_{f_\ve}$ serves as an approximate Douady-Fatou coordinate for such perturbations. This point of view is elucidated in the following lemma and later in \S \ref{ldfn}, and plays a central role in the proof of \thmref{B} in \S \ref{pfla}.  \vs  

It was first observed by X. Buff that the function $u_f$ extends to zero across the fixed points of $f$, so it gives rise to a holomorphic function $u_f: V \to \CC$. We sharpen this observation in the following statement:

\begin{lemma}\label{uw}
If $p$ is a fixed point of $f$ in $U$, then $u_f$ extends holomorphically to $p$, with $u_f(p)=0$. Moreover, if $p$ has multiplier $1$ and fixed point multiplicity $q+1 \geq 2$, then $u_f$ has a zero of order $2q$ at $p$.   
\end{lemma}

\begin{proof}
Take $p=0$ and set $\la=f'(0)$. First consider the easier case where $\la \neq 1$. Write $f(z)=zg(z)$, where $g(z)=\la+ O(z)$. By the proof of \lemref{res}, $\om_f(z)=1/(z \Log \la)+O(1)$ in a neighborhood of $0$. It follows that 
\begin{align*}
u_f(z) & = -1+(f(z)-z) \int_0^1 \om_f(z+t(f(z)-z)) \, dt \\
& = -1+\frac{z(g(z)-1)}{\Log \la} \, \int_0^1 \left(\frac{1}{z+tz(g(z)-1)}  + O(1) \right) \, dt \\
& = -1+\frac{1}{\Log \la} \Log (g(z)) + O(z) = O(z).
\end{align*}

Now suppose $\la=1$, so $f$ has the local form \eqref{EP} near $0$. Recall that the error term $\De(z)=f(z)-z$ satisfies 
\begin{align*}
\De(z) & = az^{q+1}+O(z^{q+2}) \qquad (a \neq 0) \\ 
\De'(z) & = a(q+1)z^q+O(z^{q+1}) \\ 
\De''(z) & = aq(q+1)z^{q-1}+O(z^q). 
\end{align*}
For $0 \leq t \leq 1$, 
\begin{align*}
\De(z+t\De(z)) & = \De(z) + t\De(z) \De'(z) + \frac{1}{2} t^2 \De^2(z) \De''(z) + O(z^{4q+1}) \\
\De'(z+t\De(z)) & = \De'(z) + t \De(z) \De''(z) + O(z^{3q}). 
\end{align*}
For simplicity, let us fix $z$ and write $\De=\De(z), \De'=\De'(z), \De''=\De''(z)$. By \eqref{g},
$$
\omega_f(z+t\De) = \frac{1+\frac{1}{2}(\De' + t \De \De'' + O(z^{3q}))- \frac{1}{12}(\De' + t \De \De'' + O(z^{3q}))^2}{\De + t\De \De' + \frac{1}{2} t^2 \De^2 \De'' + O(z^{4q+1})} + O(z^{2q-1}), 
$$
so 
\begin{align}\label{ghod}
& \ \ \ \ \ \ \De \, \omega_f(z+t\De) = \frac{1+\frac{1}{2}(\De' + t \De \De'' + O(z^{3q}))- \frac{1}{12}(\De' + t \De \De'' + O(z^{3q}))^2}{1 + t \De' + \frac{1}{2} t^2 \De \De'' + O(z^{3q})} + O(z^{3q}) \notag \\ 
& = \left( 1 + \tfrac{1}{2} \De' - \tfrac{1}{12} (\De')^2 + \tfrac{1}{2} t \De \De'' + O(z^{3q}) \right) \left( 1-t \De' - \tfrac{1}{2} t^2 \De \De'' + t^2 (\De')^2 + O(z^{3q}) \right)+ O(z^{3q}) \notag \\
& = 1+ \left( -t+\tfrac{1}{2} \right) \De' + \left( t^2 - \tfrac{1}{2} t -\tfrac{1}{12} \right) (\De')^2 + \left( - \tfrac{1}{2} t^2 + \tfrac{1}{2} t \right) \De \De'' + O(z^{3q}). 
\end{align}
Integrating from $t=0$ to $t=1$, the coefficients of $\De'$ and $(\De')^2$ vanish and the coefficient of $\De \De''$ gives $1/12$. Thus,  
\begin{align*}
u_f(z) & = -1+\int_0^1 \De \, \omega_f(z+t\De) \, dt = \frac{1}{12} \De \De'' + O(z^{3q}) \\
& = \frac{1}{12} a^2q(q+1) z^{2q} + O(z^{2q+1}), 
\end{align*}
as required.
\end{proof}

It follows from the above lemma that in a sufficiently small neighborhood of a fixed point of $f$, where $|u_f|<\ve$, the pair $z, f(z)$ can be lifted under $\phi_f$ to a pair $Z,F(Z)$ with the property $|F(Z)-Z-1|<\ve$. The following lemma achieves a similar control over the lift of the entire segment $[z,f(z)]$. For $z \in V^*$ and $0 \leq t \leq 1$, let $z_t:=(1-t)z+tf(z)$ and define   

\begin{equation}\label{uft}
u_{f,t}(z):=-t+\int_{[z,z_t]} \om_f(s) \, ds.   
\end{equation}  

\begin{lemma}\label{whole}
Suppose $p \in U$ is a parabolic fixed point of $f$ with multiplicity $q+1 \geq 2$. For every $0<\ve<1$ there is an $r>0$ such that $|u_{f,t}(z)|<\ve t$ for $(z,t) \in \DD(p,r) \times [0,1]$. As a result, if $z \in \DD(p,r)$ is not fixed under $f$, the curve $t \mapsto z_t$ parametrizing $[z,f(z)]$ lifts under $\phi_f$ to a curve $t \mapsto Z_t$ which satisfies $|Z_t-(Z_0+t)|<\ve t$.     
\end{lemma}

\begin{proof} 
The second claim follows immediately from the first by setting $Z_t=Z_0+t+u_{f,t}(z)$. For the first claim, let $p=0$ and $\Delta(z)=f(z)-z=O(z^{q+1})$. Take $r>0$ small enough such that   
$$
\sup_{|z|\leq r} \max \{ |\De(z)|, |\De'(z)|, |\De''(z)| \} < \tfrac{1}{2} \ve. 
$$
We can also arrange that the $O(z^{q+1})$ remainder term in \eqref{ghod} is uniformly bounded above by $\ve/2$ in the disk $\DD(0,r)$. By \eqref{ghod}, $u_{f,t}(z)$ can be expressed as 
$$
\left( -\tfrac{1}{2} t^2+\tfrac{1}{2}t \right) \De'(z) + \left( \tfrac{1}{3}t^3 - \tfrac{1}{4}t^2 -\tfrac{1}{12} t \right) (\De'(z))^2 + \left( - \tfrac{1}{6} t^3+ \tfrac{1}{4} t^2 \right) \De(z) \De''(z) + t \, O(z^{3q}).
$$
It follows that 
\begin{align*}
|u_{f,t}(z)| & < \tfrac{1}{4}t(1-t) \ve + \tfrac{1}{48}t(1-t)(4t+1) \ve^2 + \tfrac{1}{48} t^2(-2t+3) \ve^2 + \tfrac{1}{2} t \ve \\
& \leq \tfrac{1}{4} t \ve + \tfrac{5}{48} t \ve^2 + \tfrac{3}{48} t^2 \ve^2 + \tfrac{1}{2} t \ve \\
& \leq \tfrac{1}{4} t \ve + \tfrac{1}{8} t \ve + \tfrac{1}{8} t \ve + \tfrac{1}{2} t \ve =\ve t. \qedhere
\end{align*}
\end{proof}

\begin{remark}\label{uft-hol}
The family of functions $\{ u_{f,t} \}_{0 \leq t \leq 1}$ defined in \eqref{uft} extend holomorphically across the fixed points of $f$. This is trivial for $u_{f,0}=0$ and was shown in \lemref{uw} for $u_{f,1}=u_f$. For $0<t<1$, an argument similar to the proof of \lemref{uw} shows that $u_{f,t}$ takes a non-zero value at each simple fixed point and has a zero of order $q$ at each fixed point of multiplicity $q+1 \geq 2$.        
\end{remark}

\section{Near-parabolic maps and proof of \thmref{A}}\label{bfnpm}

\subsection{A convenient multiplier parameter}\label{nta} 
In order to deal with multipliers $\la$ close but not equal to $1$, it will be convenient to work with the new parameter 
$$
\La := \frac{1}{\Log \la}   
$$
instead, where $\La$ is near $\infty$. The expansion
$$
\frac{1}{\Log \la}=-\frac{1}{1-\la}+\frac{1}{2}+O(1-\la)
$$
shows that if $\lambda=\e^{1/\La}$ is the multiplier of a fixed point $p=f(p)$, then 
\begin{equation}\label{1/a=resit}
\La-\resit(f,p) = O(\La^{-1}) \qquad \text{as} \ \La \to \infty. 
\end{equation}
Thus, $\La$ blows up asymptotically the same way as the r\'esidu it\'eratif does. Evidently $p$ is repelling or attracting according as $\myre(\La)$ is positive or negative. \vs 
 
We say $\la \to 1$ {\bit non-tangentially} if $\myim (1-\la) / \myre (1-\la)$ stays bounded as $\la \to 1$. In other words, if $\la$ tends to $1$ within a sector of opening angle $< \pi$ at $1$ which is symmetric about the real axis. It is easy to see that    
\begin{equation}\label{ntaeq}
\la \to 1 \ \text{non-tangentially} \ \Longleftrightarrow \    
\frac{\myim (\La)}{\myre (\La)} = O(1) \ \text{as} \ \La \to \infty. 
\end{equation}

\subsection{Perturbations of non-degenerate parabolics} \label{pndp}
Consider a holomorphic map $g$ with a non-degenerate parabolic fixed point at $p$ whose multiplier $\la=g'(p)$ is a primitive $q$th root of unity. After an analytic change of coordinates that sends $p$ to $0$, we may assume 
$$
g(z) = \la z + O(z^2),
$$
and 
\begin{equation}\label{nfoff}
f:= g^{\circ q} = z + z^{q+1} + b z^{2q+1} + O(z^{2q+2}), \quad \text{where} \ b=\iota(f,0).\footnote{We emphasize that this normal form is a matter of convenience and not essential for Theorems A and B and their proofs.} 
\end{equation}
We fix a sufficiently small $0<r_0<1$ such that 
\begin{enumerate}
\item[$\bullet$]
$f$ is univalent in a neighborhood of the closed disk $\ov{\DD}(0,r_0)$,
\item[$\bullet$]
$0$ is the only fixed point of $f$ in $\ov{\DD}(0,r_0)$,
\item[$\bullet$]  
$\myre(f')>0$ in $\ov{\DD}(0,r_0)$. 
\end{enumerate}
Take a sequence $g_n(z)=\la_n z + O(z^2)$ of analytic perturbations of $g$. We will assume that $|\la_n| \neq 1$ for all $n$ and $f_n:=g_n^{\circ q} \to f$ uniformly on $\ov{\DD}(0,r_0)$ as $n \to \infty$. The parabolic fixed point of $g$ at $0$ bifurcates into a simple fixed point of $g_n$ at $0$ of multiplier $\la_n=g'_n(0)$ and a $q$-cycle $\sO_n = \{ p_{n,1}, \ldots, p_{n,q} \}$ of multiplier $\mu_n:=(f_n)'(p_{n,j})$ independent of $j$. The points in $\sO_n$ are asymptotically the vertices of a regular $q$-gon centered at $0$. In fact, they are the non-zero roots of the equation $f_n(z)-z=0$ near the origin which, ignoring the vanishingly small terms, has the form $(\la_n^q-1)+z^q=0$. This heuristic argument suggests that the $p_{n,j}$ are asymptotically the $q$th roots of $1-\la_n^q$. The following lemma provides a rigorous justification: \vs

\begin{lemma}\label{est2}
If $\de_n:=1-\la_n^q$, then 
$$
p_{n,j}^q = \de_n + o(\de_n) \qquad \text{as} \ n \to \infty. 
$$
\end{lemma}

\begin{proof}
Suppose 
$$
g_n(z) = \la_n z + \sum_{j=2}^q a_{n,j} z^j + a_{n,q+1} z^{q+1} + O(z^{q+2}).
$$
Here and in what follows $O(z^{q+2})$ denotes a remainder whose ratio to $z^{q+2}$ is bounded in a neighborhood of $0$ by a constant independent of $n$. Taking the $q$th iterate gives  
$$
f_n(z) = \la_n^q z + \sum_{j=2}^q A_{n,j} z^j + A_{n,q+1} z^{q+1} + O(z^{q+2}),
$$
where the coefficients $A_{n,2}, \ldots, A_{n,q}$ are universal polynomial functions of the $q$ variables $\la_n$ and $a_{n,2} \ldots, a_{n,q}$. Moreover, each of these polynomials vanishes upon substituting any $q$th root of unity for the variable $\la_n$, so it must have a factor $\de_n$. Thus, 
$$
f_n(z) = \lambda_n^q z + \de_n \sum_{j=2}^q \tilde{A}_{n,j} z^j + A_{n,q+1} z^{q+1} + O(z^{q+2}),
$$
where the $\tilde{A}_{n,j}$ are universal polynomials in $\lambda_n$ and $a_{n,2} \ldots, a_{n,q}$. As $n \to \infty$ the $\tilde{A}_{n,j}$ remain uniformly bounded in $2 \leq j \leq q$ while $A_{n,q+1} \to 1$. Now the points in the $q$-cycle $\sO_n$ are the roots of the equation   
$$
\lambda_n^q + \de_n \sum_{j=2}^q \tilde{A}_{n,j} z^{j-1} + A_{n,q+1} z^q + O(z^{q+1}) = 1,
$$ 
which can be written as 
$$
\de_n - z^q = \de_n \sum_{j=2}^q \tilde{A}_{n,j} z^{j-1} + (A_{n,q+1}-1) z^q + O(z^{q+1}) = \de_n \, o(1) + o(z^q). 
$$
This gives
\[
\frac{\de_n}{z^q} -1 = \frac{\de_n}{z^q} \, o(1) + o(1) \qquad \text{or} \qquad \frac{\de_n}{z^q} = 1+o(1). \qedhere 
\]
\end{proof}

\subsection{The lifted dynamics of $f$ and $f_n$}\label{ldfn}

As in \S \ref{nta} we introduce the parameters 
$$
\La_n := \frac{1}{\Log \lambda_n^q} \qquad \text{and} \qquad \Mu_n := \frac{1}{\Log \mu_n},
$$
so both $\La_n, \Mu_n$ tend to $\infty$ as $n \to \infty$. We may assume that $0$ and the $p_{n,j}$ are the only fixed points of $f_n$ in $\ov{\DD}(0,r_0)$. Using the normal form \eqref{nfoff} to write $f(z)=z+\De(z)$ with $\De(z)=z^{q+1}(1+bz^q+O(z^{q+1}))$, and invoking the expansion \eqref{g}, gives 
\begin{equation}\label{om}
\om(z) := \om_f(z) = \frac{1}{z^{q+1}}+\frac{\rho}{z}+e(z),
\end{equation} 
where $\rho:=\resit(f,0)=(q+1)/2-b$ and $e(z)$ is holomorphic in a neighborhood of $\ov{\DD}(0,r_0)$. Similarly, since $\om_n:=\om_{f_n}$ has simple poles at $0$ and along the $q$-cycle $\sO_n = \{ p_{n,1}, \ldots, p_{n,q} \}$, we can use \lemref{res} to write
\begin{equation}\label{omn}
\om_n(z)=\om_{f_n}(z)= \frac{\La_n}{z} + \Mu_n \sum_{j=1}^q \frac{1}{z-p_{n,j}} + e_n(z),
\end{equation} 
where $e_n(z)$ is holomorphic in a neighborhood of $\ov{\DD}(0,r_0)$. It is easy to see that $e_n \to e$ uniformly on $\ov{\DD}(0,r_0)$. In fact, in the annulus $0<r_0<|z|<r_0+2\ve$ for a fixed small $\ve$ and $n \gg 1$, we have the Laurent expansion   
\begin{align*}
\frac{\La_n}{z} + \Mu_n \sum_{j=1}^q \frac{1}{z-p_{n,j}} & = \frac{\La_n}{z} + \frac{\Mu_n}{z} \sum_{j=1}^q \frac{1}{1-p_{n,j}/z} = \frac{\La_n}{z} + \frac{\Mu_n}{z}\sum_{j=1}^q \sum_{k=0}^{\infty} \left( \frac{p_{n,j}}{z} \right)^k \\
& = (\La_n + q \Mu_n) \, \frac{1}{z} + \Mu_n \sum_{k=1}^{\infty} \Big( \sum_{j=1}^q p_{n,j}^k \Big) \ \frac{1}{z^{k+1}} 
\end{align*}
containing only negative powers of $z$. Since $\omega_n \to \omega$ uniformly on the circle $|z|=r_0+\ve$, it follows that  
$$
e_n(z) = \frac{1}{2\pi \ii} \int_{|s|=r_0+\ve} \frac{\omega_n(s)}{s-z} ds \to \frac{1}{2\pi \ii} \int_{|s|=r_0+\ve} \frac{\omega(s)}{s-z} ds = e(z)
$$ 
uniformly in $\DD(0,r_0+\ve)$. \vs

The expressions \eqref{om} and \eqref{omn} yield rather convenient formulas for the multi-valued rectifying coordinates $\phi=\phi_f$ of $\om$ and $\phi_n=\phi_{f_n}$ of $\om_n$:
\begin{align}
\phi(z) & = -\frac{1}{qz^q} + \rho \log z + E(z), \label{ZZ} \\
\phi_n(z) & = \La_n \log z + \Mu_n \sum_{j=1}^q \log(z-p_{n,j}) + E_n(z). \label{ZZn}
\end{align} 
Here $E,E_n$ are the unique primitives of $e,e_n$ with $E(0)=E_n(0)=0$. Since in each annulus $\ve<|z|<r_0+\ve$ we have the uniform convergence $\omega_n \to \omega$, it follows that for each branch of $\phi$ in this annulus there is a branch of $\phi_n$ for which $\phi_n \to \phi$ uniformly in this annulus. As for the remainder terms, $e_n \to e$ and therefore $E_n \to E$ uniformly on $\ov{\DD}(0,r_0)$. \vs

By \eqref{ZZ} the restriction of $\phi$ to the circle $|z|=r \leq r_0$ can be made uniformly close to $-1/(qz^q)+\rho \log z$ by choosing $r$ sufficiently small. Thus, given an interval $I$ of length $2\pi$, any lift $\theta \mapsto \phi(r \e^{\ii \theta})$ for $\theta \in I$ is a clockwise spiral making $q$ almost round turns having a net translation $2\pi \ii \rho$ (note that these turns have diameters comparable to $r^{-q}$, which are much larger than the fixed translation factor $2\pi \ii \rho$ when $r$ is small). Choosing a suitable branch of $\phi_n$ in \eqref{ZZn} that is uniformly close to $\phi$, it follows that the lift $\theta \mapsto \phi_n(r\e^{\ii \theta})$ for $\theta \in I$ and $n \gg 1$ is a nearby spiral also making $q$ turns, except that its net translation is $2\pi \ii \, (\La_n+q\Mu_n)$ (see \figref{spiral}). \vs 

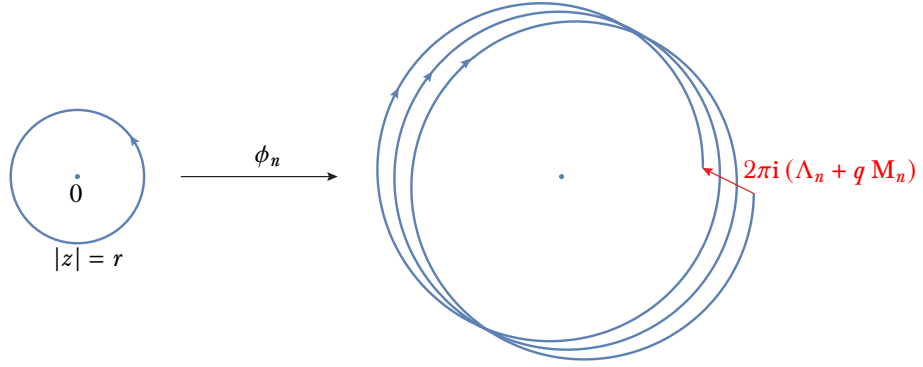
\begin{figure}[t]
	\centering
	\begin{overpic}[width=0.7\textwidth]{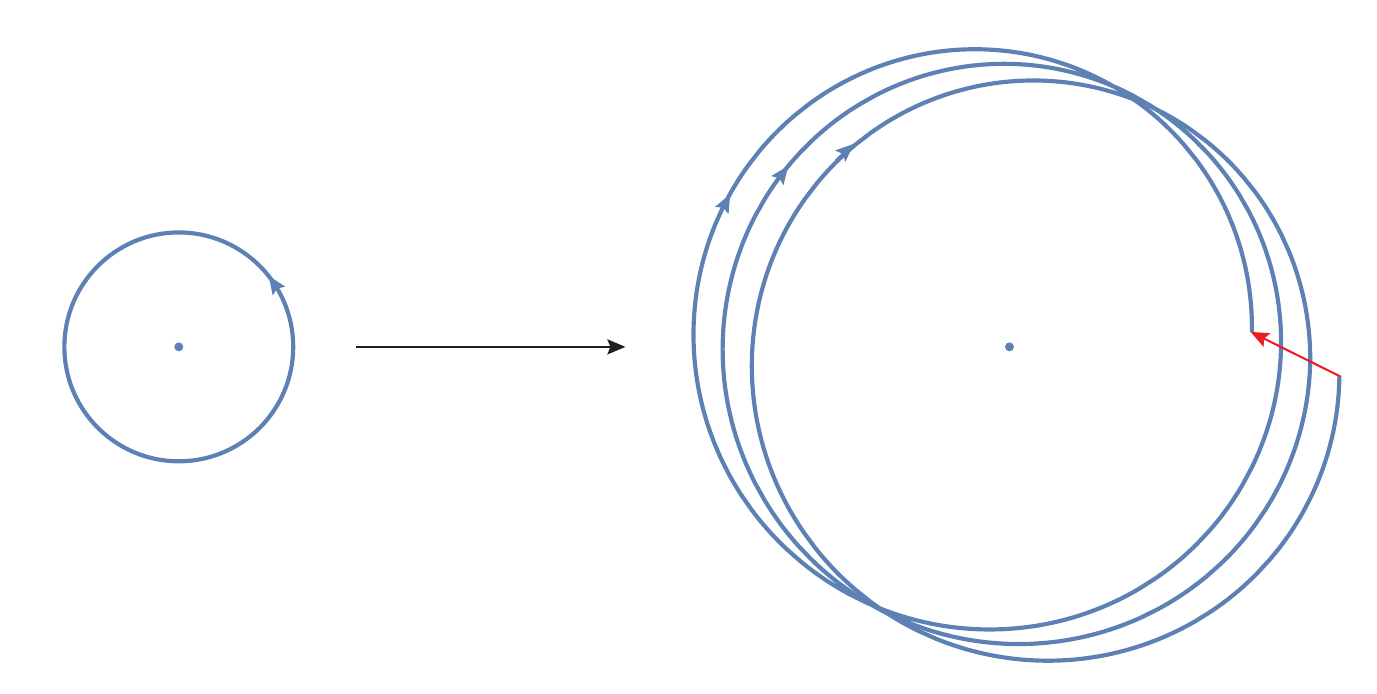}\put (12,22) {\footnotesize $0$}
\put (10,14) {\footnotesize $|z|=r$}
\put (35,27) {\footnotesize $\phi_n$}
\put (96,25) {\footnotesize \color{red}{$2\pi \ii \, (\La_n+q\Mu_n)$}}
	\end{overpic}
	\caption{\footnotesize Under any branch of the rectifying coordinate $Z=\phi_n(z)$ for $n \gg 1$ the positively oriented circle $|z|=r<r_0$ lifts to a clockwise spiral making $q$ almost round turns with a net translation $2\pi \ii(\La_n+q\Mu_n)$. In this example $q=3$. The lifted curve has strictly negative Euclidean curvature comparable to $-r^q$.}  
	\label{spiral}
\end{figure}

When $r$ is sufficiently small, the lifts of the positively oriented circle $|z|=r$ under $\phi$ and $\phi_n$ for $n \gg 1$ have strictly negative curvature. In fact, using \eqref{om} it is easy to check that 
\begin{align*}
1+\frac{z\om'(z)}{\om(z)} & = -q (1-\rho z^q + O(z^{q+1})) \\
\frac{1}{z \om(z)} & = z^q (1- \rho z^q + O(z^{q+1})).
\end{align*}
It follows from the formula \eqref{kappa} that the curvature $\kappa$ of the $\phi$-image of $|z|=r$ is
$$
\kappa = -q r^q (1+O(r^q)) < -\tfrac{q}{2} r^q
$$  
provided that $r$ is sufficiently small. Thus, the curvature of the $\phi_n$-image of $|z|=r$ is at most $-qr^q/4$ for $n \gg 1$. \vs 

We can now prepare for the proof of \thmref{A}. Choose $0<r_1<r_0/2$ such that the closed disk $\ov{\DD}(0,r_1)$ maps univalently inside $\DD(0,r_0)$ under $f$ and $f^{-1}$, hence under $f_n$ and $f_n^{-1}$ for $n \gg 1$. We may also take $r_1$ small enough to guarantee
\begin{equation}\label{rone}
\sup_{|z| \leq r_1} |f_n(z)-z| \leq \tfrac{1}{2} r_1 \qquad \text{for} \ n \gg 1. 
\end{equation}
This is possible since $f(z)-z=O(z^{q+1})$ by \eqref{nfoff}.   
We use positivity of $\myre(f')$ and $\myre(f'_n)$ throughout $\ov{\DD}(0,r_0)$ to conclude from Lemmas \ref{nofp} and \ref{uw} that the functions $u_f$ and $u_{f_n}$ are holomorphic in a neighborhood of $\ov{\DD}(0,r_1)$ and vanish at $0$ and on $\{ 0 \} \cup \sO_n$, respectively. \vs

More generally, for $0 \leq t \leq 1$ let $z_t:=(1-t)z+t f(z)$ and $z_{n,t}:=(1-t)z+t f_n(z)$ and, as in \eqref{uft}, consider the functions
\begin{align*}
u_t(z) = u_{f,t}(z)  & := -t + \int_{[z,z_t]} \om(s) \, ds \\ 
u_{n,t}(z) = u_{f_n,t}(z) & := -t + \int_{[z,z_{n,t}]} \om_n(s) \, ds, 
\end{align*}
which are holomorphic in a neighborhood of $\ov{\DD}(0,r_1)$, with $u_t(0)=0$ (see \remref{uft-hol}). 

\begin{lemma}\label{uun}
For every $\ve>0$ there is an $N \geq 1$ such that $|u_{n,t}(z)-u_t(z)| < \ve t$ whenever $(z,t) \in \ov{\DD}(0,r_1) \times [0,1]$ and $n \geq N$. 
\end{lemma}

\begin{proof}
By the maximum principle, it suffices to check the uniform estimate for $(z,t) \in \bd \DD(0,r_1) \times [0,1]$. Given $\ve>0$, find $N_1$ such that $\sup_{|z|=r_1}|f_n(z)-f(z)|<\ve$ and therefore $\sup_{|z|=r_1} |z_{n,t}-z_t|< \ve t$ for $n \geq N_1$. The closed annulus $K := \{ z: r_1/2 \leq |z| \leq 3r_1/2 \}$ is contained in $\DD(0,r_0)$ and \eqref{rone} shows that there is an $N_2$ such that $[z,z_{n,t}] \cup [z,z_t] \cup [z_t,z_{n,t}] \subset K$ for $(z,t) \in \ov{\DD}(0,r_1) \times [0,1]$ and $n \geq N_2$. Finally, there is an $N_3$ such that $\sup_{z \in K} |\omega_n(z)-\omega(z)|< \ve$ for $n \geq N_3$. It follows that for $(z,t) \in \ov{\DD}(0,r_1) \times [0,1]$ and $n \geq \max \{N_1, N_2, N_3 \}$,    
\begin{align*}
|u_{n,t}(z)-u_t(z)| & \leq \left| \int_{[z,z_{n,t}]} (\om_n(s)-\om(s)) \, ds \right| + \left| \int_{[z_t,z_{n,t}]} \om(s) \, ds \right| \\
& \leq \tfrac{1}{2} r_1 t \cdot \ve + \ve t \cdot \sup_{s \in K} |\omega(s)| = C(r_1) \ \ve t. \qedhere  
\end{align*}
\end{proof}

\begin{proof}[Proof of \thmref{A}]
Given $\ve>0$ we can find $0<r<r_1$ such that for $n \gg 1$, 
$$
|u_{n,t}(z)|<\ve t \qquad \text{if} \ (z,t) \in \DD(0,r)\times [0,1]. 
$$
This follows by combining Lemmas \ref{whole} and \ref{uun}. In turn, it proves \eqref{vet} if we set $Z_{n,t}:=Z_{n,0}+t+u_{n,t}(z)$. The similar bound \eqref{vett} can be established by repeating the argument for the inverse map $f_n^{-1}$ and taking $r$ smaller if necessary, or by noting that 
$$
Z_{n,1-t}-Z_{n,1}+t = t - \int_{[z_{n,1-t},z_{n,1}]} \om_n(s) \, ds
$$ 
and proving in a manner similar to \lemref{whole} that the last expression has magnitude bounded by $\ve t$. 
\end{proof}

We end this section with a homotopy result that will be used in \S \ref{pfla} and \S \ref{tma}. It is a corollary of the following   

\begin{lemma}\label{homot}
There is an $0<r_2<r_1$ such that the following holds for $n \gg 1$: If three orbit points $z_0, z_1:=f_n(z_0), z_2:=f_n(z_1)$ in $\DD(0,r_2)$ are not fixed by $f_n$, then $[z_1,z_2]$ and $f_n([z_0,z_1])$ are homotopic in $\DD(0,r_1) \sm (\{ 0 \} \cup \sO_n )$. 
\end{lemma}

\begin{proof}
Choose $0<r_2<r_1/2$ small enough so that for $n \gg 1$,
$$
|f'_n(z)-1|<\tfrac{1}{10}, \quad \text{hence} \quad 
\tfrac{9}{10}\, |z-w| \leq |f_n(z)-f_n(w)|\leq \tfrac{11}{10}\, |z-w| 
$$
whenever $z,w \in \DD(0,2r_2)$. Take $z_0,z_1,z_2$ in $\DD(0,r_2)$ as above and let $d:=|z_1-z_2|$. We claim that there is no fixed point of $f_n$ in the $d/3$-neighborhood of the segment $[z_1,z_2]$. In fact, since $\myre(f'_n)>9/10$ throughout $\DD(0,r_2)$, the same estimate as in the proof of \lemref{nofp} shows that $|f_n(w)-w| \geq 9d/10$ for every $w \in [z_1,z_2]$. If there were a fixed point $p$ of $f_n$ and some $w \in [z_1,z_2]$ with $|w-p|<d/3$, then $|p|<r_2+d/3<2r_2$ and therefore 
$$
|f_n(w)-w| \leq |f_n(w)-f_n(p)|+|p-w| \leq \tfrac{21}{10} |w-p|< \tfrac{7}{10} d,
$$      
which would be a contradiction. Now for every $w \in [z_0,z_1]$, 
$$
f_n(w)=z_1 + \int_{[z_0,w]} f'_n(\zeta) \, d\zeta = z_1 + (w-z_0) + \int_{[z_0,w]} (f'_n(\zeta)-1) \, d\zeta,
$$
where the last integral has absolute value $<|w-z_0|/10$. Thus, both $f_n([z_0,z_1])$ and $[z_1,z_2]$ are contained in the convex region 
$$
\bigcup_{\zeta \in [z_1,2z_1-z_0]} \DD \big( \zeta,\tfrac{1}{10}|\zeta-z_1| \big). 
$$
Since $|z_0-z_1|/10 \leq d/9$, this region is contained in the $d/3$-neighborhood of $[z_1,z_2]$ and therefore is free from fixed points of $f_n$. The result follows since by the choice $r_2<r_1/2$, the $d/3$-neighborhood of $[z_1,z_2]$ is contained in $\DD(0,r_1)$.         
\end{proof} 

\begin{corollary}\label{joo}
Let $r_2$ be as in \lemref{homot}, $0<r<r_2$ and $n \gg 1$. Take a (finite or infinite) backward orbit $z_j:=f_n^{-j}(z_0)$ for $0 \leq j < k \leq \infty$ in $\DD(0,r) \sm (\{ 0 \} \cup \sO_n )$. Let $\eta_j$ be a curve in $\DD(0,r) \sm (\{ 0 \} \cup \sO_n )$ joining $z_{j-1}$ to $z_j$ such that $f_n^{-1}(\eta_{j-1})=\eta_j$ for every $2 \leq j < k$. If $\eta_1$ and $[z_0,z_1]$ are homotopic in $\DD(0,r) \sm (\{ 0 \} \cup \sO_n )$, so are $\eta_j$ and $[z_{j-1},z_j]$ for every $1 \leq j < k$.   
\end{corollary}

\begin{proof}
Let us write $\sim$ for homotopy. By \lemref{homot}, $[z_0,z_1] \sim f_n([z_1,z_2])$ in $\DD(0,r_1) \sm (\{ 0 \} \cup \sO_n )$, so by our assumption $\eta_1 \sim f_n([z_1,z_2])$ in $\DD(0,r_1) \sm (\{ 0 \} \cup \sO_n )$. Taking the image of this homotopy under $f_n^{-1}$ then gives $\eta_2 \sim [z_1,z_2]$ in $\DD(0,r_0) \sm (\{ 0 \} \cup \sO_n )$ and therefore in $\DD(0,r) \sm (\{ 0 \} \cup \sO_n )$ since the latter is a deformation retraction of the former. Now proceed inductively. 
\end{proof}

\section{Invariant curves and proof of \thmref{B}}\label{apphl}

\subsection{Non-tangential multiplier approach}\label{ntma}

We continue working in the setup of \S \ref{pndp}, namely, a non-degenerate parabolic map $g(z)=\la z+O(z^2)$ with the multiplier $\la$ a primitive $q$th root of unity, normalized so that $f:=g^{\circ q}$ takes the form \eqref{nfoff}, and a sequence of perturbations $g_n(z)=\la_n z+O(z^2)$ with $|\la_n| \neq 1$ such that $g_n \to g$ uniformly on $\ov{\DD}(0,r_0)$. The fixed point $0$ of $g$ bifurcates into a simple fixed point of $g_n$ at $0$ of multiplier $\la_n$ and a nearby $q$-cycle $\sO_n = \{ p_{n,1}, \ldots, p_{n,q} \}$ of multiplier $\mu_n$. The points of $\sO_n$ are asymptotically the vertices of a regular $q$-gon, as asserted in \lemref{est2}. \vs  

We will be interested in the case where the multipliers tend to $1$ non-tangentially. As before, set $f_n:=g_n^{\circ q}$, $\La_n = 1/\Log \lambda_n^q$ and $\Mu_n := 1/\Log \mu_n$. By continuity of the r\'esidu it\'eratif (see \eqref{resitcont}),
$$
\resit(f_n,0)+\sum_{j=1}^q \resit(f_n,p_{n,j}) \to \resit(f,0), 
$$  
which, in view of \eqref{1/a=resit}, yields 
\begin{equation}\label{sumind}
\La_n + q \Mu_n \to \rho = \frac{q+1}{2}-b.   
\end{equation}
as $n \to \infty$. This shows $\La_n$ and $-q\Mu_n$ have similar asymptotic behaviors. It follows from \eqref{ntaeq} that $\lambda_n^q \to 1$ non-tangentially if and only if $\mu_n \to 1$ non-tangentially. The non-tangential approach implies that $0$ and $\sO_n$ cannot be both repelling or both attracting. In fact, if $|\lambda_n|>1, |\mu_n|>1$ for $n \gg 1$, then $\myre(\La_n)$ and $\myre(\Mu_n)$ are positive and therefore they both tend to $+\infty$. This would be impossible by \eqref{sumind}. Similarly, if $|\lambda_n|<1, |\mu_n|<1$ for $n \gg 1$, then $\myre(\La_n)$ and $\myre(\Mu_n)$ are negative and therefore they both tend to $-\infty$, which would be impossible again. It follows that if $\lambda_n^q, \mu_n \to 1$ non-tangentially, then, after passing to a subsequence, we may assume $|\lambda_n|<1, |\mu_n|>1$, or $|\lambda_n|>1, |\mu_n|<1$ for $n \gg 1$. \vs

The assumption that $\lambda_n^q, \mu_n \to 1$ non-tangentially also puts a restriction on the $q$-cycle $\sO_n$. In fact, under this assumption there is a small $\theta_0>0$ such that one of the following holds (compare \figref{sectors}): \vs

\begin{figure}[]
	\centering
	\begin{overpic}[width=0.75\textwidth]{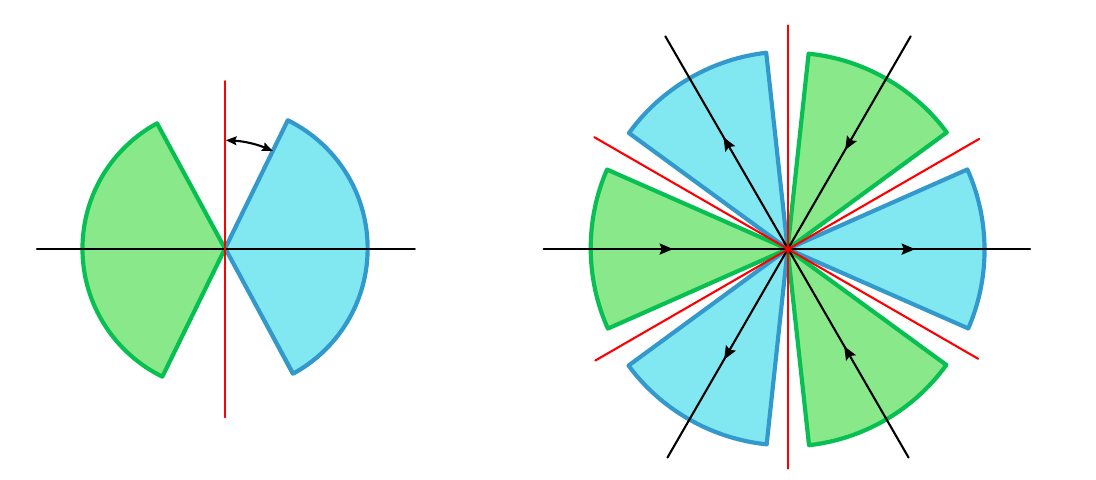}
\put (22,32) {\footnotesize $\theta_0$}
\put (61,31) {\footnotesize $S^+_1$}
\put (66,7) {\footnotesize $S^+_2$}
\put (84,23) {\footnotesize $S^+_3$}
\put (75,34) {\footnotesize $S^-_1$}
\put (56,18) {\footnotesize $S^-_2$}
\put (80,10) {\footnotesize $S^-_3$}
	\end{overpic}
\caption{\sl Left: The sectors to which $1-\la_n^q$ belongs in a non-tangential approach. Right: The repelling sectors $S^+_j$ (in blue) and the attracting sectors $S^-_j$ (in green) to which the $q$-cycle $\sO_n$ belongs. Here $q=3$.}
\label{sectors}
\end{figure}
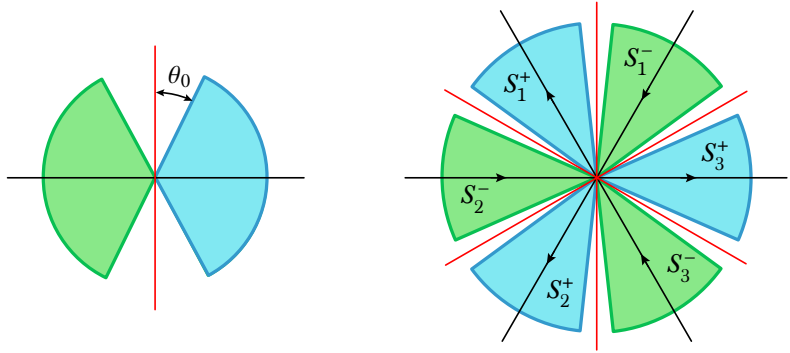

$\bullet$ {\bf Case A.} $1-\la_n^q$ belongs to the sector $\{ r \e^{\ii \theta}: 0<r<1, |\theta|<\pi/2-\theta_0 \}$ for $n \gg 1$. Then $|\la_n|<1, |\mu_n|>1$. It follows from \lemref{est2} that each point of the repelling $q$-cycle $\sO_n$ belongs to precisely one of the sectors 
\begin{equation}\label{s+}
S^+_j := \left\{ r \e^{\ii \theta}: 0<r<1, \left| \theta - \frac{2j\pi}{q} \right| < \frac{1}{q} \left( \frac{\pi}{2} - \frac{\theta_0}{2} \right) \right\} \qquad (1 \leq j \leq q). 
\end{equation}
We label the points of $\sO_n$ such that $p_{n,j} \in S^+_j$. Note that the mid-axis of $S_j^+$ is the $j$th repelling direction of the parabolic point $0$ for $f$. \vs 

$\bullet$ {\bf Case B.} $1-\la_n^q$ belongs to the sector $\{ r \e^{\ii \theta}: 0<r<1, |\theta-\pi|<\pi/2-\theta_0 \}$ for $n \gg 1$. Then $|\lambda_n|>1, |\mu_n|<1$. It follows from \lemref{est2} that each point of the attracting $q$-cycle $\sO_n$ belongs to precisely one of the sectors 
\begin{equation}\label{s-}
S^-_j := \left\{ r \e^{\ii \theta}: 0<r<1, \left| \theta - \frac{(2j-1)\pi}{q} \right| < \frac{1}{q} \left( \frac{\pi}{2} - \frac{\theta_0}{2} \right) \right\} \qquad (1 \leq j \leq q). 
\end{equation}
We similarly label the points of $\sO_n$ such that $p_{n,j} \in S^-_j$. Now the mid-axis of $S_j^-$ is the $j$th attracting direction of the parabolic point $0$ for $f$. 
 
\subsection{Proof of \thmref{B}}\label{pfla}

The assumption $\la_n^q, \mu_n \to 1$ non-tangentially means that $\myim(\La_n)/\myre(\La_n)$ and $\myim(\Mu_n)/\myre(\Mu_n)$ are bounded sequences (see \eqref{ntaeq}). Thus, there is a $\delta>0$ such that the unimodular complex numbers 
$$
\alpha_n := \ii \frac{\La_n}{|\La_n|} \qquad \text{and} \qquad \beta_n := \ii \frac{\Mu_n}{|\Mu_n|} 
$$
have their arguments in $[-\pi+\delta,-\delta] \cup [\delta, \pi-\delta]$. Fix an $\ve$ so that $0<\ve<\sin \delta$. Take a sufficiently small radius $r>0$ subject to the following conditions: 
\begin{enumerate}
\item[$\bullet$]
\thmref{A} holds for $r$ and the above choice of $\ve$. \vs 
\item[$\bullet$]
$0<r<r_2$, so \corref{joo} holds for $r$. \vs 
\item[$\bullet$]
The curvature of the lift of the positively oriented circle $|z|=r$ under the rectifying coordinate $\phi_n$ for $\om_n$ is strictly negative, say $<-q r^q/4$, for $n \gg 1$ (see \S \ref{ldfn}). \vs 
\item[$\bullet$]
$\ga \cap \ov{\DD}(0,r)$ is contained in one of the repelling sectors of \eqref{s+}, say in $S^+_q$, the sector with the real segment $]0,1[$ as its mid-axis. This is possible because $\ga$ is asymptotic to a repelling direction of the parabolic point $0$ of $f$. Take $t_0=t_0(r)<0$ such that $|\ga(t_0)|=r$ and $\ga(]-\infty,t_0[) \subset S^+_q \cap \DD(0,r)$. \vs
\item[$\bullet$]
For the branch of $\phi$ in \eqref{ZZ} with $\log r$ real, 
$|\phi(r\e^{\ii \theta})+r^{-q} \e^{-\ii q \theta}/q| \leq r^{-q}$ for all $\theta \in [-\pi,\pi]$. Since $|\ga(t_0)-r|=o(r)$, we can guarantee $|\phi(\ga(t_0))+r^{-q}/q| \leq 2r^{-q}$.\footnote{We can do much better, but these bounds will be enough for our purposes.}\vs 
\end{enumerate}

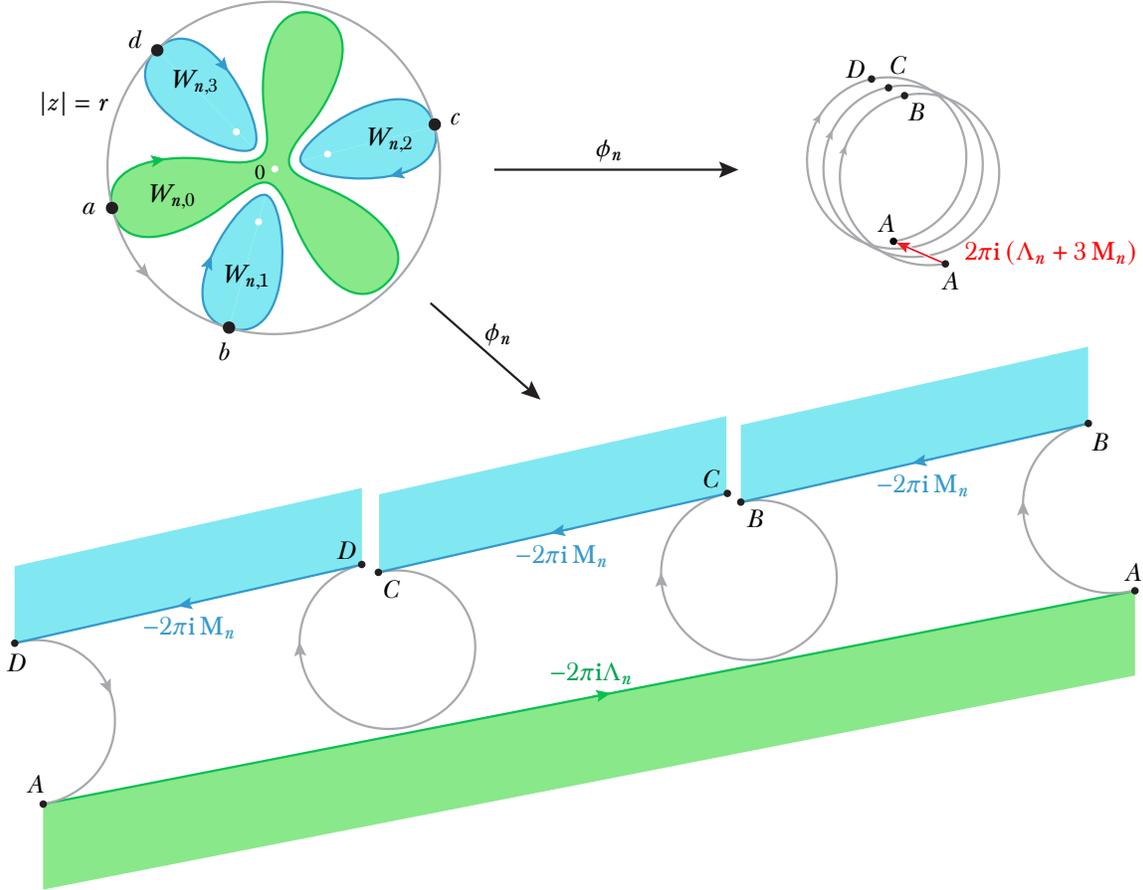
\begin{figure}[t]
	\centering
	\begin{overpic}[width=\textwidth]{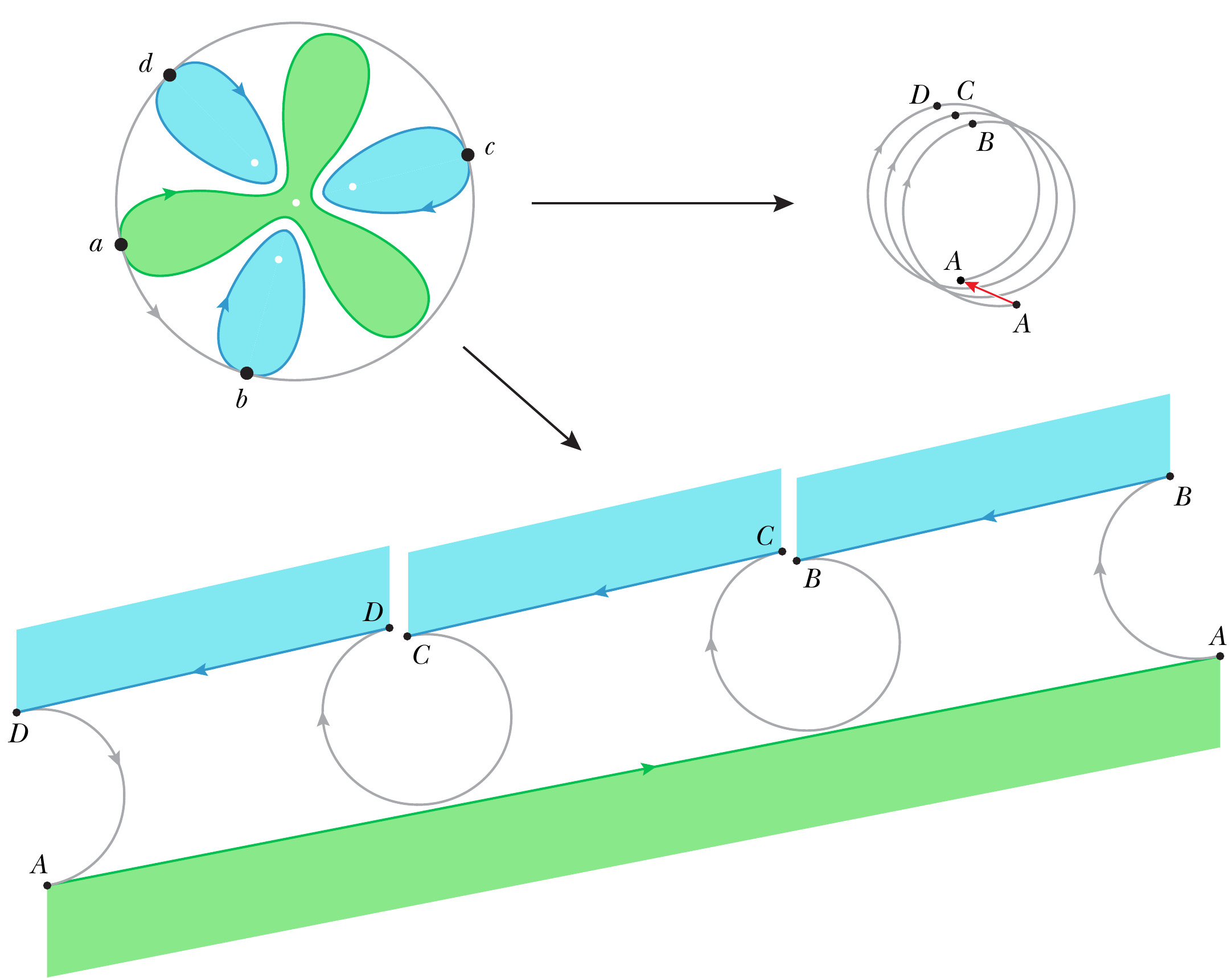}
		\put (84.2,55) {\footnotesize \color{red}{$2\pi \ii \, (\La_n+3\Mu_n)$}}
      \put (48,18.3) {\footnotesize \color{mygreen}{$-2\pi \ii\La_n$}}
      \put (76.5,34.7) {\footnotesize \color{myblue}{$-2\pi \ii\Mu_n$}}
      \put (45,28.5) {\footnotesize \color{myblue}{$-2\pi \ii\Mu_n$}}
      \put (12.5,22.3) {\footnotesize \color{myblue}{$-2\pi \ii\Mu_n$}}
		\put (22.3,62) {\tiny $0$}
		\put (3.5,68) {\footnotesize $|z|=r$}
		\put (52,64) {\footnotesize $\phi_n$}
		\put (42.3,48) {\footnotesize $\phi_n$}
      \put (13,60) {\footnotesize $W_{n,0}$}
      \put (19.5,53) {\footnotesize $W_{n,1}$}
      \put (32,65) {\footnotesize $W_{n,2}$}
      \put (15,70) {\footnotesize $W_{n,3}$}
	\end{overpic}
	\caption{\footnotesize Illustration of two lifts under the rectifying coordinate $\phi_n$ for $q=3$ and the case $|\la_n|<1, |\mu_n|>1$. Top left: The disk $\DD(0,r)$ and the canonical neighborhoods $W_{n,j}$ for $0 \leq j \leq 3$. Top right: Under $\phi_n$ the positively oriented circle $|z|=r$, starting and ending at $a$, lifts to a spiral of negative curvature making three almost round turns of size comparable to $r^{-3}$, with a net translation $2\pi \ii(\La_n+3\Mu_n)$. Bottom: The null-homotopic loop, starting and ending at $a$, consisting of the negatively oriented boundary $\bd W_{n,0}$ followed by $ab, \bd W_{n,1}, bc, \bd W_{n,2}, cd, \bd W_{n,3}, da$, lifts to a closed curve. Moving around this closed curve which starts and ends at the lower left $A$ involves one translation by $-2\pi \ii \La_n$ (the lift of $\bd W_{n,0}$) and three translations by $-2\pi \ii \Mu_n$ (the lifts of $\bd W_{n,j}$ for $1 \leq j \leq 3$), interjected by almost round arcs which are translated pieces of the top right spiral. The three blue half-planes obtained by lifting $W_{n,1}, W_{n,2}, W_{n,3}$ belong to different sheets of the Riemann surface of $\phi_n$.}  
\label{sheets}
\end{figure}

Let $W_{n,0} \subset \DD(0,r)$ be the canonical neighborhood of $p_{n,0}:=0$ for the rotated vector field $\alpha_n \, \chi_n$. Here $\chi_n=\chi_{f_n}$ is the Buff vector field associated with $f_n$. Similarly for $1 \leq j \leq q$ let $W_{n,j} \subset \DD(0,r)$ be the canonical neighborhood of $p_{n,j}$ for the rotated vector field $\beta_n \, \chi_n$. By \corref{jcs} $W_{n,0}, \ldots, W_{n,q}$ have pairwise disjoint closures and the boundary of each meets the circle $|z|=r$ tangentially. Recall that each punctured neighborhood $W_{n,j} \sm \{ p_{n,j} \}$ lifts under $Z=\phi_n(z)$ to a half-plane $H_{n,j}$ of the form $\{ Z: \myim(Z/\alpha_n) > \con \}$ if $j=0$ and $\{ Z: \myim(Z/\beta_n) > \con \}$ if $1 \leq j \leq q$. From this picture it also follows that the boundary of $W_{n,0}$ meets the circle $|z|=r$ tangentially in at most $q$ points while the boundary of $W_{n,j}$ for $1 \leq j \leq q$ is tangent to this circle at a unique point. In fact, since the curve $\theta \mapsto \phi_n(r\e^{\ii \theta})$ has strictly negative curvature, its unit tangent vector turns monotonically clockwise $q$ times on every half-open interval $I$ of length $2\pi$. Thus, there are precisely $q$ angles $\theta \in I$ where a half-plane with the slope determined by $\alpha_n$ or $\beta_n$ can be externally tangent to this curve (compare \figref{sheets}).%
\vs

Every canonical neighborhood $W_{n,j}$ corresponding to a repelling fixed point is backward invariant under $f_n$ for $n \gg 1$. To see this, take $z \in W_{n,j} \sm \{ p_{n,j} \}$ and let $Z=\phi_n(z) \in H_{n,j}$ be a lift of $z$. The point $p_{n,j}$ being repelling means the corresponding $\alpha_n$ or $\beta_n$ has positive imaginary part, so its argument is in $[\delta,\pi-\delta]$. Since $\sin^{-1}\ve<\delta$ by our choice of $\ve$, the cone $\sC(Z):=\sC^-_\ve(Z)$ defined in \eqref{coneC} is contained in $H_{n,j}$. By \thmref{A} the lift of $[z,f_n^{-1}(z)]$ under $\phi_n$ that starts at $Z$ and ends at $F_n^{-1}(Z)$ is contained in $\sC(Z)$, hence in $H_{n,j}$, for $n \gg 1$. Projecting down by $\psi_n=\phi_n^{-1}$, we obtain $[z,f_n^{-1}(z)] \subset W_{n,j} \sm \{ p_{n,j} \}$, proving backward invariance. \vs  

A similar argument shows that every $W_{n,j}$ corresponding to an attracting fixed point is forward-invariant under $f_n$ for $n \gg 1$. \vs 

For the rest of the proof we fix our ideas by assuming $|\la_n|<1,|\mu_n|>1$, so the origin is attracting while the $q$-cycle $\sO_n = \{ p_{n,1}, \ldots, p_{n,q} \}$ is repelling. This corresponds to Case A at the end of \S \ref{ntma} (the argument in Case B where $|\la_n|>1,|\mu_n|<1$ is completely similar). Thus, $p_{n,j}$ stays in the sector $S^+_j$ for $1 \leq j \leq q$ and $n \gg 1$. Set $S=S^+_q \cap \DD(0,r)$ and recall that there is a $t_0<0$ such that $|\ga(t_0)|=r$ and $\ga(]-\infty,t_0[) \subset S$. Let us fix a large positive integer $N_0$ whose suitable size will be determined later (the important point is that $N_0$ is independent of $n$). Since $\ga_n \to \ga$ uniformly on compact subsets of $]-\infty,0]$, we can find a sequence $t_n \to t_0$ such that $|\ga_n(t_n)|=r$ and $\ga_n([t_n-N_0,t_n[) \subset S$ for $n \gg 1$. Take a single-valued branch of $\phi$ in some neighborhood of $\ga([t_0-1,t_0[)$ with $|\phi(\ga(t_0))+r^{-q}/q|<2r^{-q}$, and find single-valued branches of $\phi_n$ such that $\phi_n \to \phi$ uniformly in this neighborhood. This branch of $\phi_n$ can be continued analytically in the punctured domains $S \sm \{ p_{n,q} \}$ and $W_{n,q} \sm \{ p_{n,q} \}$, both with monodromy $Z \mapsto Z+2\pi \ii \, \Mu_n$. Thus, there is a planar domain $\Om_n$ containing the half-plane $H_{n,q}=\phi_n(W_{n,q} \sm \{ p_{n,q} \})$ and an infinite cyclic covering map $\psi_n:=\phi_n^{-1} : \Om_n \to (W_{n,q} \cup S) \sm \{ p_{n,q} \}$ with the deck group generated by the translation $Z \mapsto Z+ 2\pi \ii \, \Mu_n$.\footnote{The choice of $(W_{n,q} \cup S) \sm \{ z_n \}$ and its lift $\Om_n$ is made for our particular argument here. There are other choices of punctured neighborhoods of $p_{n,q}$ whose lift give much larger domains for the map $\psi_n$.} Note that the line $\bd H_{n,q}$ is tangent to the curve $\theta \mapsto \phi_n(r\e^{\ii \theta})$ for $-\pi \leq \theta \leq \pi$,  with the tangency point having modulus comparable to $r^{-q}$ (see \figref{lifts}). \vs   

Now consider the lifted curves $\Ga:=\phi \circ \ga$ and $\Ga_n:=\phi_n \circ \ga_n$ obtained by analytic continuation of the single-valued branches of $\phi, \phi_n$ along $\ga$ and $\ga_n$. Thus, $\Ga$ is defined on $]-\infty,t_0]$ and $\Ga_n$ is defined at least on $[t_n-N_0,t_n]$. The fundamental segment $J_n:=\Ga_n([t_n-1,t_n[)$ converges to the fundamental segment $\Ga([t_0-1,t_0[)$ as $n \to \infty$, so it remains in a fixed neighborhood of $\Ga(t_0)$, hence in the disk $\DD(-r^{-q}/q,R)$ for some $R>0$ independent of $n$. Setting $Z_0:=-r^{-q}/q+R/\ve$, it follows that $J_n \subset \sC(Z_0)$ for $n \gg 1$. The same argument as before based on $\sin^{-1}\ve < \de$ now shows that the `tail' of the cone $\sC(Z_0)$ is eventually contained in the half-plane $H_{n,q}$, that is, there exists an $x_0 \in \RR$ such that 
$$
\sC(Z_0) \cap \{ Z: \myre(Z)<x_0 \} \subset H_{n,q} \qquad \text{for} \ n \gg 1 
$$ 
(compare \figref{lifts}). Recall that $F_n^{-1}$ is the lifted dynamics of $f_n^{-1}$ obtained by integrating along the segments $[z,f_n^{-1}(z)]$, as described in \S \ref{ldfn}. By \corref{joo}, for $n \gg 1$, if $t_n-N_0+1 \leq t \leq t_n$ then the segment $[\ga_n(t-1),\ga_n(t)]$ and the restriction of $\ga_n$ to $[t-1,t]$ are homotopic in $\DD(0,r) \sm (\sO_n \cup \{ 0 \})$. It follows that       
$$
F_n^{-1}(\Ga_n(t))=\Ga_n(t-1) \qquad \text{if} \ t_n-N_0+1 \leq t \leq t_n.
$$  
Since $\ga_n([t_n-N_0,t_n[) \subset S$ for $n \gg 1$, 
the first $N_0-1$ iterates of $J_n$ under $F_n^{-1}$ are defined and $\bigcup_{j=0}^{N_0-1} F_n^{-j}(J_n) = \Ga_n([t_n-N_0,t_n[)$ is contained in $\Om_n \cap \sC(Z_0)$ by \thmref{A}. Moreover, $\myre(Z) \leq R_0:=R-r^{-q}/q$ for all $Z \in J_n$, so inductively $\myre(Z) \leq R_0-j(1-\ve)$ for all $Z \in F_n^{-j}(J_n)$. If we chose $N_0$ large enough so that $R_0-N_0(1-\ve) <x_0$, it would follow that all further iterates $F_n^{-j}(J_n)$ for $j \geq N_0$ are defined and contained in $H_{n,q} \cap \sC(Z_0)$. Since $\sin^{-1} \ve < \de$, the Euclidean distance from $\Ga_n(t)$ to $\bd H_{n,q}$ grows linearly in $t$ as $t \to -\infty$. Projecting back to the dynamical plane by $\psi_n$, this implies that the $|\om_n|$-distance from $\ga_n(t)$ to $\bd W_{n,q}$ tends to $+\infty$ as $t \to -\infty$, which shows $\lim_{t \to -\infty} \ga_n(t)=p_{n,q}$ (alternatively, we could use the fact that $f_n^{-1}: W_{n,q} \to W_{n,q}$ strongly contracts the hyperbolic metric to reach the same conclusion). This proves part (i) of \thmref{B}. \vs

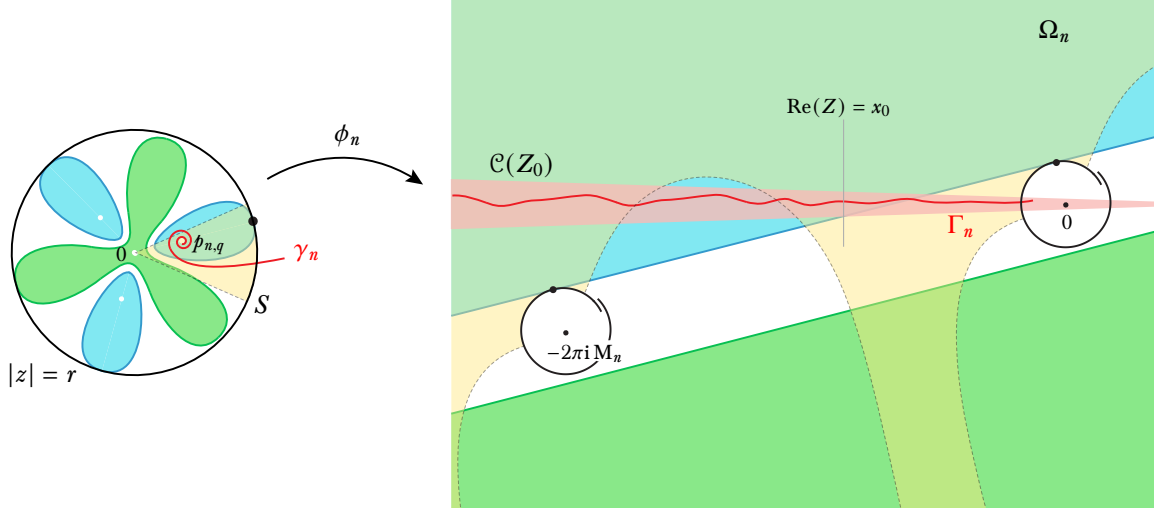
\begin{figure}[t]
	\centering
	\begin{overpic}[width=\textwidth]{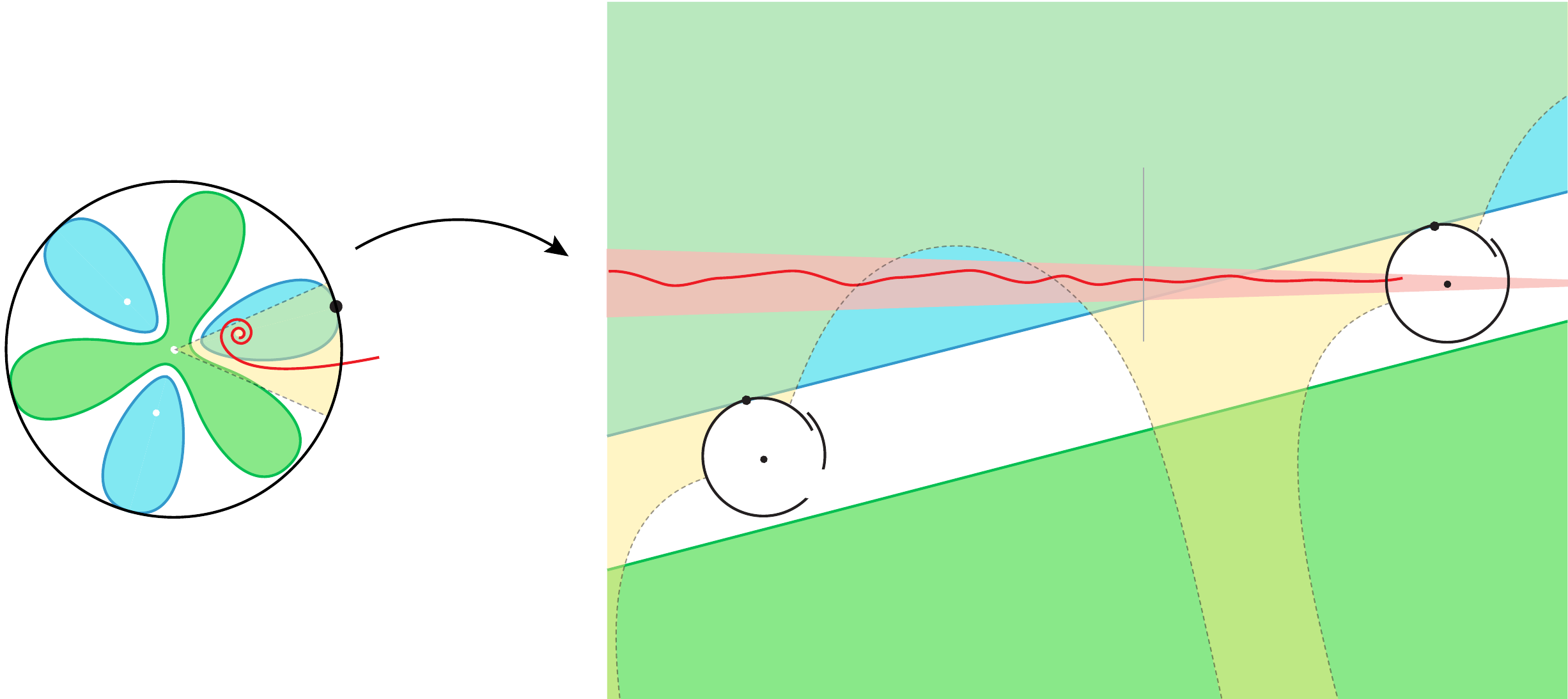}
		\put (9.5,21.5) {\tiny $0$}
      \put (92,24.3) {\tiny $0$}
		\put (0,11) {\footnotesize $|z|=r$}
      \put (42,29.3) {\footnotesize $\sC(Z_0)$}
		\put (28.5,32) {\footnotesize $\phi_n$}
      \put (25,22) {\color{red}{\footnotesize $\ga_n$}}
      \put (82,24) {\color{red}{\footnotesize $\Ga_n$}}
      \put (16.2,22.8) {\tiny $p_{n,q}$}
		\put (68,34.5) {\tiny $\myre(Z)=x_0$}
      \put (90,41) {\footnotesize $\Om_n$}
     \put (21.5,17) {\footnotesize $S$}
     \put (47,13) {\tiny $-2\pi \ii \Mu_n$}
	\end{overpic}
	\caption{\footnotesize Illustration of the proof of \thmref{B}. The domain $\Om_n$ is the union of the blue half-plane $H_{n,q}$ and the yellow region $\phi_n(S \sm \{ p_{n,q} \})$.}  
	\label{lifts}
\end{figure}

Let us now prove part (ii). Since $\lim_{t \to -\infty} \ga(t)=0$ and $\ga_n \to \ga$ uniformly on compact subsets of $]-\infty,0]$, it suffices to show that for every radius $r'$ with $0<r'<r$ there is a $t^*=t^*(r')<0$ such that $|\ga_n(t)|<r'$ for $t<t^*$ and $n \gg 1$. Since $\om_n \to \om$ uniformly on the closed annulus $A:=\{ z: r' \leq |z| \leq r \}$ as $n \to \infty$, the metrics $|\om_n|$ are uniformly bounded on $A$ for $n \gg 1$. Hence, there is a constant $C=C(r')>0$ such that 
\begin{equation}\label{diam}
\diam_{|\om_n|}(A)<C \qquad \text{for} \ n \gg 1.
\end{equation}
Recall from part (i) that $Z \in \sC(Z_0)$ and $\myre(Z)<x_0$ imply $Z \in H_{n,q}$. Using $\sin^{-1}\ve<\de$ once more, we see that there is an $x_1<x_0$ independent of $n$ such that the Euclidean distance from any $Z \in \sC(Z_0)$ to $\bd H_{n,q}$ is $>2C$ whenever $\myre(Z)<x_1$. As in the proof of part (i), the lifted curve $t \mapsto \Ga_n(t)$ stays in $\sC(Z_0)$ for $t<t_n$. If we choose the integer $N_1$ large enough so that $R_0-N_1(1-\ve)<x_1$, it follows that $\myre(\Ga_n(t))<x_1$ for $t<t_n-N_1$. This implies that the Euclidean distance from $\Ga_n(t)$ to $\bd H_{n,q}$ is $>2C$ for $t<t_n-N_1$. Projecting back to the dynamical plane by $\psi_n$, it follows that the $|\om_n|$-distance from $\ga_n(t)$ to $\bd W_{n,q}$ is $>2C$ for $t<t_n-N_1$. Since $\bd W_{n,q}$ meets the circle $|z|=r$ and $t_n \to t_0$, the choice of $C$ satisfying \eqref{diam} shows that $|\ga_n(t)|<r'$ whenever $t<t_0-N_1-1$ and $n \gg 1$. This completes the proof.

\subsection{Tangential multiplier approach}\label{tma}

Without the assumption of non-tangential multiplier approach in \thmref{B} the $f_n$-invariant curves $\ga_n$ may take longer and longer to land at a repelling point near $0$ or they may have a ``near miss'' without landing and eventually leave the neighborhood of $0$. In such cases the Hausdorff limit of $\ga_n$ will be strictly larger than $\ga$. This phenomenon can already be observed in polynomial dynamics where the Hausdorff limit of external rays along a convergent sequence $P_n \to P$ can contain invariant arcs in the filled Julia set of $P$. These are the heteroclinic and homoclinic arcs studied in \cite{PZ1}. \vs

For an invariant embedded arc that has a near miss and eventually leaves the neighborhood of $0$ one can use the theory of parabolic implosions to show that the arc must pass through one of the $q$ ``gates'' between $0$ and the points of the bifurcated cycle $\sO_n$ (compare \figref{gate}). Below we show how the setup developed in this section gives a simple proof of this fact. We continue working in the situation of \thmref{B}, with $\ga, \ga_n$ being invariant embedded arcs under $f, f_n$ such that $\ga$ lands at the parabolic point $0$ and $\ga_n \to \ga$ uniformly on compact subsets of $]-\infty,0]$, but the multipliers $\la_n^q,\mu_n$ are now tending to $1$ tangentially. The assumption that $\ga, \ga_n$ are embedded arcs is for convenience and sufficient for standard applications. However, one can properly formulate and prove a similar statement for invariant curves or even more general continua.     

\begin{figure}[t]
	\centering
	\begin{overpic}[width=0.45\textwidth]{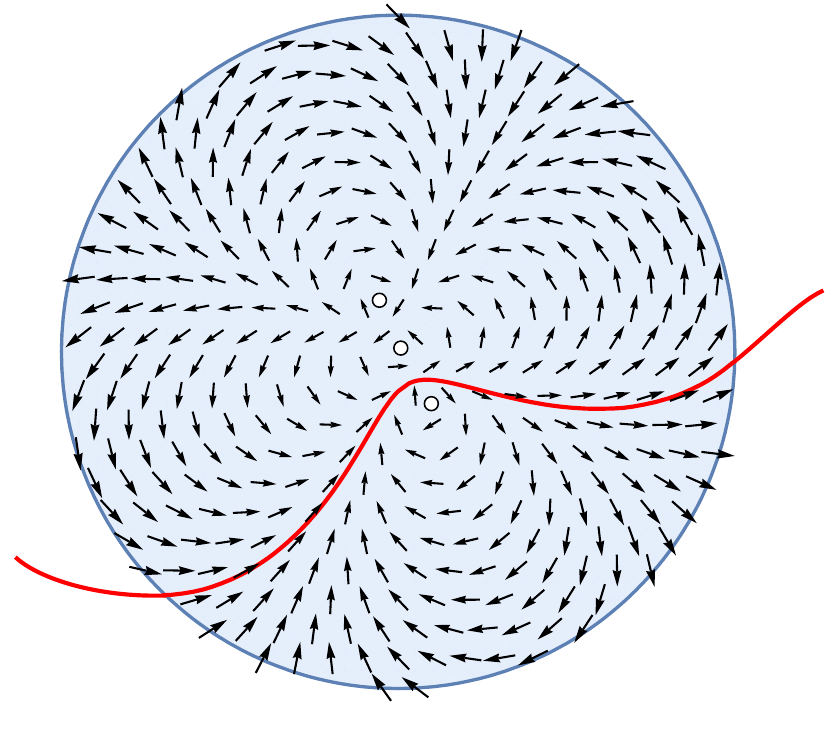} 
\put (96,47) {\footnotesize \color{red}{$\ga_n$}}
	\end{overpic}
	\caption{\footnotesize The $f_n$-invariant curve $\ga_n$ has a near miss and leaves the neighborhood of $0$ without landing at any of the three fixed points of $f_n$. In doing so $\ga_n$ has to pass through one of the two ``gates'' between $0$ and the pair of bifurcated fixed points. In this example $q=2$ and the arrows show the Buff vector field $\chi_{f_n}$.}  
	\label{gate}
\end{figure}
    
\begin{theorem}[Near miss arcs pass through the gates]
Fix a small $r>0$ and suppose $\ga_n$ eventually leaves $\DD(0,r)$ in the sense that there are sequences $-\infty<s_n<t_n \leq 0$ such that $\ga_n(t) \in \DD(0,r)$ if and only if $s_n<t<t_n$. Then, for $n \gg 1$, $\ga_n$ must separate the fixed points of $f_n$ in $\DD(0,r)$. In other words, both components of $\DD(0,r) \sm \ga_n(]s_n,t_n[)$ meet $\{ 0 \} \cup \sO_n$. 
\end{theorem}      

\begin{proof}
We may assume $r$ is small enough so that \thmref{A} holds for, say, $\ve=1/2$ and also $r<r_2$ so \corref{joo} holds. Evidently $|\ga_n(t_n)|=|\ga_n(s_n)|=r$. Since $\ga$ lands at $0$ and $\ga_n \to \ga$ uniformly on compact subsets of $]-\infty,0]$, $t_n$ is bounded and $s_n \to -\infty$. As in the proof of \thmref{B}, take a single-valued branch of $\phi_n$ mapping $\ga_n([t_n-1,t_n[)$ into an $R$-neighborhood of $-r^{-q}/q$ and continue it analytically along $\ga_n$ to obtain a curve $\Ga_n=\phi_n \circ \ga_n$ defined on $[s_n,t_n]$. By \corref{joo}, $F_n^{-1}(\Ga_n(t))=\Ga_n(t-1)$ for $s_n+1 \leq t \leq t_n$. By \thmref{A}, $\Ga_n$ stays in the cone $\sC(Z_0)=\sC^-_\ve(Z_0)$, where $Z_0=-r^{-q}/q+R/\ve$ as before. Let $J_n=\Ga_n([t_n-1,t_n[)$ and choose the integer $N_n$ such that $s_n+N_n \in [t_n-1,t_n[$. Then $\Ga_n(s_n) \in F_n^{-N_n}(J_n)$ so $\myre(\Ga_n(s_n)) \leq R_0 - N_n(1-\ve)$. In particular, $\myre(\Ga_n(s_n)) \to -\infty$ as $n \to \infty$.  \vs

However, if $\ga_n$ does not separate the fixed points of $f_n$ in $\DD(0,r)$, the arc $\ga_n([s_n,t_n])$ is homotopic in $\ov{\DD}(0,r) \sm ( \{ 0 \} \cup \sO_n)$ to an arc of the circle $|z|=r$ joining $\ga_n(t_n)$ to $\ga_n(s_n)$. Letting $\ga_n(t_n)=r \e^{\ii \theta_n}$ with $\theta_n \in \, ]\pi,\pi]$, it follows that the lift $\Ga_n(s_n)$ must be on the same spiral curve $\theta \mapsto \phi_n(r\e^{\ii \theta}), \ \theta_n-\pi \leq \theta \leq \theta_n+\pi$ as $\Ga_n(t_n)$ is. This implies $\myre(\Ga_n(s_n))$ being bounded below by $\con r^{-q}$, which is a contradiction.  
\end{proof}


\begin{thebibliography}{AAAAA}

\bibitem [\bf{BE}]{BE} X.~Buff and A.~Epstein, {\it A parabolic Yoccoz inequality}, Fundamenta Mathematicae {\bf 172} (2002) 249-289.   

\bibitem [\bf{C}]{C} A.~Ch\'eritat, {\it Sur l'implosion parabolique, la taille des disques de Siegel \& une conjecture de Marmi, Moussa et Yoccoz}, Habilitation thesis, 2008, available at \url{https://www.math.univ-toulouse.fr/~cheritat}

\bibitem [\bf{D}]{D} A. Douady, {\it Does a Julia set depend continuously on the polynomial?}, in Complex Dynamical Systems: The Mathematics Behind the Mandelbrot and Julia Sets, Proceedings of Symposia in Applied Mathematics {\bf 49}, 1994, pp. 91-138. 

\bibitem[\bf{HSD}]{HSD} M. Hirsch, S. Smale, and R. Devaney, {\it Differential Equations, Dynamical Systems, and an Introduction to Chaos}, 3rd ed, Academic Press, 2012.

\bibitem[\bf{IY}]{IY} Y. Ilyashenko and S. Yakovenko, {\it Lectures on Analytic Differential Equations}, Graduate Studies in Mathematics 86, American Mathematical Society, 2007.

\bibitem[\bf L]{L} P. Lavaurs, {\it Syst\`emes dynamiques holomorphes: Explosion de points p\'eriodiques}, Th\`ese, Universit\'e de Paris-Sud, 1989.

\bibitem [\bf{Mc}]{Mc} C. McMullen, {\it Hausdorff dimension and conformal dynamics II: Geometrically finite rational maps}, Comment. Math. Helv. {\bf 75} (2000) 535-593.   

\bibitem[\bf M1]{M1} J. Milnor, {\it Dynamics in One Complex Variable}, 3rd ed., Annals of Mathematics Studies {\bf 160}, Princeton University Press, 2006.

\bibitem[\bf M2]{M2} J. Milnor, {\it Periodic Orbits, externals rays and the Mandelbrot set: an expository account}, pages 171-229 in Collected Papers of John Milnor, vol. VII, AMS Publications, 2014. 

\bibitem[\bf{O}]{O} R. Oudkerk, {\it The parabolic implosion for $f_0(z) = z + z^{\nu+1} + {\mathcal O}(z^{\nu+2})$}, Doctoral dissertation, Warwick University, 1999.  

\bibitem [\bf{PZ1}]{PZ1} C. L. Petersen and S. Zakeri, {\it Hausdorff limits of external rays: the topological picture}, Proc. London Math. Soc. {\bf 128} (2024) e12598.

\bibitem [\bf{PZ2}]{PZ2} C. L. Petersen and S. Zakeri, {\it Lemon limbs in the cubic connectedness locus}, 2025, to appear. 

\bibitem[\bf{Sh}]{Sh} M. Shishikura, {\it Bifurcation of parabolic fixed points}, pages 325-364 in The Mandelbrot Set, Theme and Variations, edited by Tan Lei, Cambridge University Press, 2000. 

\bibitem[\bf{St}]{St} K. Strebel, {\it Quadratic Differentials}, Springer, 1984.

\end{thebibliography}
\end{document}